\documentclass[10pt]{article}

\usepackage[margin=1.5in]{geometry}

\usepackage[T1]{fontenc}
\usepackage[utf8]{inputenc}
\usepackage[USenglish, UKenglish, english]{babel}
\usepackage{amssymb}
\usepackage{amsthm}
\usepackage{amsmath}
\usepackage[mathscr]{euscript}
\usepackage{enumitem}
\usepackage{graphicx}
\usepackage{MnSymbol}
 \let\mathscr\relax
\usepackage[scr]{rsfso}
\usepackage{tikz-cd}
\usepackage{tikz}
\usetikzlibrary{matrix,arrows,decorations.pathmorphing}
\usepackage{hyperref}
\usepackage{marginnote}
\usetikzlibrary{arrows.meta}

\usepackage{comment}

\usepackage[raggedright]{titlesec}

\theoremstyle{definition}
\newtheorem{defin}{Definition}[section]
\theoremstyle{definition}

\theoremstyle{plain}
\newtheorem{theo}[defin]{Theorem}
\theoremstyle{plain}
\newtheorem{prop}[defin]{Proposition}
\theoremstyle{plain}
\newtheorem{lem}[defin]{Lemma}
\theoremstyle{plain}
\newtheorem{cor}[defin]{Corollary}
\theoremstyle{definition}

\theoremstyle{definition}

\theoremstyle{definition}

\theoremstyle{plain}
\newtheorem{conj}[defin]{Conjecture}
\theoremstyle{definition}

\theoremstyle{definition}
\newtheorem{hyp}[defin]{Hypothesis}

\theoremstyle{definition}
\newtheorem*{defin*}{Definition}
\theoremstyle{definition}
\newtheorem*{ex*}{Example}
\theoremstyle{plain}
\newtheorem*{theo*}{Theorem}
\theoremstyle{plain}
\newtheorem*{prop*}{Proposition}
\theoremstyle{plain}
\newtheorem*{lem*}{Lemma}
\theoremstyle{plain}
\newtheorem*{cor*}{Corollary}
\theoremstyle{definition}
\newtheorem*{rmk*}{Remark}
\theoremstyle{definition}
\newtheorem*{exe*}{Exercise}

\theoremstyle{plain}

\usepackage{parskip}

\makeatletter
\def\thm@space@setup{%
  \thm@preskip=\parskip \thm@postskip=0pt
}
\makeatother

\theoremstyle{plain}
\newtheorem{theoA}{Theorem}[section]

\theoremstyle{plain}
\newtheorem{corA}[theoA]{Corollary}

\numberwithin{equation}{section}

\usepackage{libertine}

\usepackage{graphicx}
\usepackage{xcolor}
\usepackage{titlesec}
\usepackage{microtype}

\definecolor{myblue}{RGB}{0,82,155}

\titleformat{\chapter}[display]
  {\normalfont\bfseries\color{black}}
  {\centering
    {}    
    {\fontsize{50}{50}\bfseries\selectfont\thechapter}
  }
  {5pt}
  {\titlerule[1.5pt]\vskip1pt\titlerule\vskip5pt\centering\Huge\selectfont}

\setlist[enumerate]{label=(\roman*)}

\def\Br{{\rm Br}}

\def\Bl{{\rm Bl}}

\def\bl{{\rm bl}}

\def\h{{\rm ht}}

\def\irr{{\rm Irr}}

\def\proj{{\rm Proj}}

\def\ker{{\rm Ker}}

\def\syl{{\rm Syl}}

\def\tr{{\rm tr}}

\def\ind{{\rm Ind}}

\def\n{{\mathbf{N}}}

\def\c{{\mathbf{C}}}

\def\z{{\mathbf{Z}}}

\def\O{{\mathbf{O}}}

\def\d{{\mathbb{D}}} 

\def\e{{\mathbb{E}}} 

\def\C{{\mathcal{C}}}

\def\cl{{\mathfrak{Cl}}}

\makeatletter
\newcommand{\uset}[3][0ex]{%
  \mathrel{\mathop{#3}\limits_{
    \vbox to#1{\kern-7\ex@
    \hbox{$\scriptstyle#2$}\vss}}}}
\makeatother

\newcommand{\wt}[1]{\widetilde{#1}} 

\newcommand{\wh}[1]{\widehat{#1}}

\newcommand{\ws}[1][1.5]{
  \mathrel{\scalebox{#1}[1]{$\sim$}}
}

\newcommand{\iso}[1]{\ws_{#1}}

\usepackage{sectsty}

\allsectionsfont{\centering}

\makeatletter
\def\blfootnote{\gdef\@thefnmark{}\@footnotetext}
\makeatother

\title{
{\huge\bf Character Triple Conjecture for $p$-Solvable Groups}\\
\author{\Large Damiano Rossi}
\date{}
\blfootnote{\emph{$2010$ Mathematical Subject Classification:} $20$C$20$ ($20$C$25$).
\\
\emph{Key words and phrases:} Character Triple Conjecture, Dade's Projective Conjecture, $p$-solvable groups.
\\
This paper has been written during the author's PhD at the Bergische Universit\"at Wuppertal funded by the research training group \textit{GRK2240: Algebro-geometric Methods in Algebra, Arithmetic and Topology} of the DFG. The author would like to thank his advisor Britta Sp\"ath for a careful reading of the manuscript.
}
}

\begin{document}

\renewcommand{\thetheoA}{\Alph{theoA}}

\renewcommand{\thecorA}{\Alph{corA}}

\selectlanguage{english}

\maketitle

\begin{center}
\emph{To the memory of Carlo Casolo}
\end{center}

\begin{abstract}
In this paper, we prove the Character Triple Conjecture for $p$-solvable groups. This is a conjecture proposed by Sp\"ath during the reduction process of Dade's Projective Conjecture to quasisimple groups (see \cite{Spa17}). In addition,  as suggested by Isaacs and Navarro in \cite{Isa-Nav02}, we take into account the $p$-residue of characters.
\end{abstract}

\section{Introduction}

Global-local counting conjectures play a major role in modern representation theory of finite groups. Amongst them are the McKay Conjecture \cite{McK72} and its blockwise version, known as the Alperin-McKay Conjecture \cite{Alp76}, the Alperin Weight Conjecture \cite{Alp87} and a series of conjectures proposed by Dade in \cite{Dad92}, \cite{Dad94} and \cite{Dad97} that imply all the above mentioned conjectures. Dade's aim was to find a version of his conjecture strong enough to hold for every finite group if proved for all nonabelian simple groups. Unfortunately, such a reduction theorem has never been published. The first step towards the solution of the global-local conjectures has been achieved by Isaacs, Malle and Navarro in \cite{Isa-Mal-Nav07} where the McKay Conjecture was reduced to a stronger statement for simple groups. Inspired by this result, other reduction theorems have been proved (see \cite{Nav-Tie11}, \cite{Spa13II}, \cite{Spa13I}, \cite{Spa17} and \cite{Nav-Spa-Val}). However, contrary to Dade's philosophy, all the reduction theorems appeared so far reduce a certain statement for arbitrary finite groups to a much stronger statement for quasisimple groups. 

Although these stronger statements, known as inductive conditions, have been originally thought for (quasi)simple groups, they can be stated for arbitrary finite groups. Then, going back to Dade's plan, by proving the inductive condition for simple groups it should be possible to obtain, not only the original conjecture, but even the inductive condition itself for every finite group. This was done in \cite{Nav-Spa14I} for the Alperin--McKay Conjecture. 

In \cite{Spa17} Sp\"ath introduced the Character Triple Conjecture and showed that Dade's Projective Conjecture holds for every finite group if her conjecture holds for all quasisimple groups. Therefore Sp\"ath's conjecture plays the role of inductive condition for Dade's Projective Conjecture. Following \cite{Nav-Spa14I}, we would like to show that the Character Triple Conjecture holds for every finite group if it holds for all quasisimple groups. To prove such a reduction theorem, it is necessary to study the structure of a minimal counterexample. As for the above mentioned reductions, the first step in this direction is to show that such a counterexample cannot be $p$-solvable. This result was in preparation when the paper \cite{Spa17} was published. Unfortunately, the project was later abandoned and no proof has appeared. With permission of Sp\"ath, we take over the project and prove the Character Triple Conjecture for $p$-solvable groups in this paper.

More precisely, let $G$ be a finite group and fix a prime $p$. For any $d\geq 0$, we denote by $\irr^d(G)$ the set of irreducible characters $\chi\in\irr(G)$ with $p$-defect equal to $d$. Let $\mathfrak{N}(G)$ be the set of $p$-chains of $G$ starting with $\O_p(G)$. If $\d=\{D_0<D_1<\dots<D_n\}\in\mathfrak{N}(G)$, we denote by $|\d|$ the integer $n$, called the length of $\d$. This yields a partition of $\mathfrak{N}(G)$ into the set $\mathfrak{N}(G)_+$ of $p$-chains of even length and the set $\mathfrak{N}(G)_-$ of $p$-chains of odd length. Notice that $G$ acts by conjugation on the set of $p$-chains and let $G_\d$ be the stabilizer in $G$ of the chain $\d$. For $\epsilon\in\{+,-\}$ and $B$ a $p$-block of $G$, define
$$\C^d(B)_\epsilon:=\left\lbrace(\d,\vartheta)\enspace\middle|\enspace \d\in\mathfrak{N}(G)_\epsilon,\vartheta\in\irr^d(G_\d),\bl(\vartheta)^G=B\right\rbrace,$$
where $\bl(\vartheta)$ is the unique block of $G_\d$ containing $\vartheta$ and $\bl(\vartheta)^G$ is the block of $G$ obtain via Brauer's induction (this is defined by \cite[Lemma 3.2]{Kno-Rob89}). We denote by $\overline{(\d,\vartheta)}$ the $G$-orbit of $(\d,\vartheta)\in\C^d(B)_\epsilon$ and by $\C^d(B)_\epsilon/G$ the set of $G$-orbits. Now, our main result can be stated as follows.

\begin{theoA}
\label{thm:Main theorem}
Let $G$ be a finite $p$-solvable group with $\O_p(G)\leq \z(G)$ and consider a $p$-block $B$ of $G$ with noncentral defect groups. Suppose that $G\unlhd A$ and denote by $A_B$ the stabilizer of $B$ in $A$. Then, for every $d\geq 0$, there exists an $A_B$-equivariant bijection
$$\Omega:\C^d(B)_+/G \to \C^d(B)_-/G$$
such that
$$\left(A_{\d,\vartheta},G_\d,\vartheta\right)\iso{G}\left(A_{\e,\chi},G_{\e},\chi\right),$$
in the sense of Definition \ref{def:N-block isomorphic character triples}, for every $(\d,\vartheta)\in\C^d(B)_+$ and $(\e,\chi)\in\Omega(\overline{(\d,\vartheta)})$. 
\end{theoA} 

Recall that for $\chi\in\irr(G)$, the $p$-residue of $\chi$ is the nonnegative integer $r(\chi):=|G|_{p'}/\chi(1)_{p'}$. Following ideas of Isaacs and Navarro \cite{Isa-Nav02}, we include the $p$-residue of characters into the picture. However, we do not consider more technical refinements involving Galois automorphisms (see \cite{Nav04}) or $p$-local Schur indices (see \cite{Tur08II}) as done by Turull in \cite{Tur17} for Dade's Conjecture.

\begin{theoA}
\label{thm:Main theorem residues}
There exists a bijection $\Omega$ satisfying the conditions of Theorem \ref{thm:Main theorem} and such that 
$$r(\vartheta)\equiv \pm r(\chi)\pmod{p},$$
for every $(\d,\vartheta)\in\C^d(B)_+$ and some $(\e,\chi)\in\Omega(\overline{(\d,\vartheta)})$.
\end{theoA}

As a corollary to our results, we show that Dade's Extended Projective Conjecture \cite[4.10]{Dad97}, with the Isaacs-Navarro refinement, holds for every $p$-solvable group.

\begin{corA}
\label{cor:Dade's Extended Projective Conjecture}
Dade's Extended Projective Conjecture with the Isaacs-Navarro refinement holds for every $p$-solvable group.
\end{corA}

\begin{proof}
This follows from Theorem \ref{thm:Main theorem residues} and \cite[Proposition 6.4]{Spa17}.
\end{proof}

The paper is structured as follows: in the next section, we establish some notation, we give the main definitions used in the paper and we prove some preliminary results. In the third section, we construct an equivariant defect preserving character bijection lying above the Glauberman correspondence that is well behaved with respect to $N$-block isomorphic character triples (see Definition \ref{def:N-block isomorphic character triples}). In the fourth section, we show how to construct $N$-block isomorphic character triples by using the Fong correspondence. Finally in the last section, we use the previously obtained results to prove Theorem \ref{thm:Main theorem} and Theorem \ref{thm:Main theorem residues}. This is done by inspecting the structure of a minimal counterexample.

\section{Preliminaries and notation}

We use standard notation from representation theory of finite groups as in \cite{Isa76}, \cite{Nag-Tsu89} and \cite{Nav98}. All groups considered in the sequel are assumed to be finite. For notational convenience, whenever necessary we denote the normalizer $\n_G(H)$ simply by $G_H$, for every $H\leq G$. 

Let $\irr(G)$ the set of ordinary irreducible characters. If $H\leq G$ and $\chi\in\irr(G)$, then $\irr(\chi_H)$ is the set of irreducible characters of $H$ which appear as constituents of the restricted character $\chi_H$. Moreover, for $\vartheta\in\irr(H)$, we denote by $\irr(G\mid \vartheta)$ the set of irreducible constituents of the induced character $\vartheta^G$. For $N\unlhd G$ and $\vartheta\in\irr(N)$ we denote by $G_\vartheta$ the stabilizer of $\vartheta$ in $G$ and by $\irr_G(N)$ the set of $G$-invariant irreducible characters of $N$. Then $(G,N,\vartheta)$ is a character triple if $N\unlhd G$ and $\vartheta\in\irr_G(N)$.

Fix a prime $p$. For $\chi\in\irr(G)$, there exist unique nonnegative integers $d(\chi)$ and $r(\chi)$, called respectively the $p$-defect and the $p$-residue of $\chi$, such that $r(\chi)p^{d(\chi)}=|G|/\chi(1)$ with $r(\chi)$ coprime to $p$. For any $d\geq 0$, we denote by $\irr^d(G)$ the set of irreducible characters $\chi\in\irr(G)$ such that $d(\chi)=d$. We denote by $\Bl(G)$ the set of $p$-blocks of $G$ and by $\Bl(G\mid b)$ the set of all blocks of $G$ covering $b$, where $N\unlhd G$ and $b\in\Bl(N)$. Let $G_b$ be the stabilizer of $b$ in $G$. If $H\leq G$ and $b\in\Bl(H)$, then $b^G$ is, when it is defined, the block obtained via Brauer induction. For $\chi\in\irr(G)$, the block of $G$ that contains $\chi$ is $\bl(\chi)$. Let $\delta(B)$ be the set of defect groups of the block $B$ and $d(B)$ be its defect. If $D$ is a $p$-subgroup of $G$, then $\Bl(G\mid D)$ is the set of blocks of $G$ with $D\in\delta(B)$.

For the notion of projective representation, we refer to \cite{Nag-Tsu89}, \cite{Nav18} and \cite{Spa18}. We denote by $\proj(G\mid \alpha)$ the set of projective representations of $G$ with factor set $\alpha$. Moreover, if $N\unlhd G$, then we can consider a representation of $G/N$ as a representation of $G$ that is constant on $N$-cosets by the usual inflation process. If $\mathcal{P}$ is a projective representation associated with a character triple $(G,N,\vartheta)$ (see \cite[Definition 1.7]{Spa18} and \cite[Definition 5.2]{Nav18}), then $\mathcal{P}$ yields a central extension of $G$ (see \cite[Theorem 5.6]{Nav18} and \cite[Theorem 1.12]{Spa18}). This is a standard construction and we will make use of it without further comment.

In \cite[Definition 6.3]{Spa17} a new equivalence relation on character triples was introduced. This will be of fundamental importance in what follows and, for completeness, we include the definition.

\begin{defin}[$N$-block isomorphic character triples]
\label{def:N-block isomorphic character triples}
Let $(H_1,M_1,\vartheta_1)$ and $(H_2,M_2,\vartheta_2)$ be two character triples and let $N$ be a group. We say that the two character triples are $N$-block isomorphic, and write
$$\left(H_1,M_1,\vartheta_1\right)\iso{N}\left(H_2,M_2,\vartheta_2\right),$$
if the following conditions are satisfied:
\begin{enumerate}
\item $N\unlhd NH_1=NH_2=:G$, $M_1=H_1\cap N$ and $M_2=H_2\cap N$. We denote the canonical isomorphisms by $l_i:H_i/M_i\to G/N$ and by $i:=l_2^{-1}\circ l_1:H_1/M_1\to H_2/M_2$;
\item For $i=1,2$, there exists a defect group $D_i\in\delta(\bl(\vartheta_i))$ such that $\c_G(D_i)\leq H_i$. In particular $\c_G(N)\leq H_1\cap H_2$;
\item For $i=1,2$, there exists a projective representation $\mathcal{P}_i\in\proj(H_i\mid \alpha_i)$ associated with $(H_i,M_i,\vartheta_i)$ such that $\alpha_1(x,y)=\alpha_2(i(x),i(y))$, for every $x,y\in H_1/M_1$, and with the property that $\mathcal{P}_{1,\c_G(N)}$ and $\mathcal{P}_{2,\c_G(N)}$ are associated with the same scalar function (see the comments preceding \cite[Definition 2.7]{Spa18} and \cite[Definition 10.14]{Nav18}). In this case, there exists a strong isomorphism of character triple $(i,\sigma):(H_1,M_1,\vartheta_1)\to (H_2,M_2,\vartheta_2)$ (see \cite[Problem 11.13]{Isa76}) given by
\begin{align*}
\sigma_{J_1}:\irr(J_1\mid \vartheta_1)&\to \irr(J_2\mid \vartheta_2)
\\
\tr(\mathcal{Q}_{J_1}\otimes \mathcal{P}_{1,J_1})&\mapsto \tr(\mathcal{Q}_{J_2}\otimes \mathcal{P}_{2,J_2}),
\end{align*}
for every $N\leq J\leq G$, where $J_i:=J\cap H_i$ and $\mathcal{Q}\in\proj(J/N\mid \alpha^{-1}_{J\times J})$. Here $\alpha$ is the factor set of $G/N$ corresponding to $\alpha_i$ via the isomorphism $l_i:H_i/M_i\to G/N$ (see \cite[Theorem 3.3]{Spa17});
\item For $N\leq J\leq G$, we have
$$\bl(\psi)^J=\bl(\sigma_{J_1}(\psi))^J$$
for every $\psi\in\irr(J_1\mid \vartheta_1)$. Observe in this situation that block induction is well defined (see \cite[Lemma 3.5]{Spa17}).
\end{enumerate}
In this situation, we say that $(\mathcal{P}_1,\mathcal{P}_2)$ is associated with $(H_1,M_1,\vartheta_1)\iso{N}(H_2,M_2,\vartheta_2)$, or that $(H_1,M_1,\vartheta_1)\iso{N}(H_2,M_2,\vartheta_2)$ is given by $(\mathcal{P}_1,\mathcal{P}_2)$. Whenever we want to specify the pair $(\mathcal{P}_1,\mathcal{P}_2)$ we write $\sigma^{(\mathcal{P}_1,\mathcal{P}_2)}_{J_1}$ instead of simply $\sigma_{J_1}$. 
\end{defin}

The above definition gives a relation on character triples that extends the relation $\geq_b$ introduced in \cite[Definition 4.2]{Spa18} (see also \cite{Nav-Spa14I}). In fact, observe that $(G,N,\chi)\geq_b (H,M,\vartheta)$ if and only if $(G,N,\chi)\iso{N}(H,M,\vartheta)$. On the other hand, the notion of block isomorphism of character triples given in \cite[Definition 3.6]{Nav-Spa14I} is slightly different, and in some sense more restrictive, from the relation $\geq_b$ (see \cite[Remark 4.3 (c)]{Spa18}). For this reason, we will not refer to results of \cite{Nav-Spa14I} which involve directly block isomorphism of character triples.

Having defined $N$-block isomorphic character triples, we can now introduce Sp\"ath's Character Triple Conjecture (\cite[Conjecture 6.3]{Spa17}). For a central $p$-subgroup $Z$ of $G$, we denote by $\mathfrak{N}(G,Z)$ the set of normal $p$-chains of $G$ starting with $Z$. These are the chains $\d=\{Z=D_0<D_1<\dots<D_n\}$ of $p$-subgroups of $G$ with the property that each $D_i$ is normal in the largest subgroup $D_n$. We denote by $|\d|$ the integer $n$, called the length of $\d$. The set $\mathfrak{N}(G,Z)$ is partitioned into the set $\mathfrak{N}(G,Z)_+$ of $p$-chains of even length, and the set $\mathfrak{N}(G,Z)_-$ of $p$-chains of odd length. The group $G$ acts by conjugation on $\mathfrak{N}(G,Z)$ and we denote by $G_\d=\bigcap_i\n_G(D_i)$ the stabilizer in $G$ of the chain $\d$. Finally, let $B$ be a block of $G$ and, for $\epsilon\in\{+,-\}$ and $d\geq 0$, define $\C^d(B,Z)_\epsilon$ to be the set of pairs $(\d,\vartheta)$ with $\d\in\mathfrak{N}(G,Z)_\epsilon$ and $\vartheta\in\irr^d(G_\d)$ satisfying $\bl(\vartheta)^G=B$. Again, the group $G$ acts on $\C^d(B,Z)_\epsilon$. We denote by $\overline{(\d,\vartheta)}$ the $G$-orbit of $(\d,\vartheta)\in\C^d(B,Z)_\epsilon$ and by $\C^d(B,Z)_\epsilon/G$ the set of $G$-orbits. 

\begin{conj}[Sp\"ath's Character Triple Conjecture]
\label{conj:Character Triple Conjecture}
Let $G$ be a finite group, $Z\leq \z(G)$ be a $p$-subgroup and consider $B\in\Bl(G)$ with defect groups strictly larger than $Z$. Suppose that $G\unlhd A$. Then, for every $d\geq 0$, there exists an $\n_A(Z)_B$-equivariant bijection
$$\Omega:\C^d(B,Z)_+/G \to \C^d(B,Z)_-/G$$
such that 
$$\left(A_{\d,\vartheta},G_\d,\vartheta\right)\iso{G}\left(A_{\e,\chi},G_{\e},\chi\right),$$
for every $(\d,\vartheta)\in\C^d(B,Z)_+$ and some $(\e,\chi)\in\Omega(\overline{(\d,\vartheta)})$.
\end{conj}

By \cite[Lemma 3.8 (c)]{Spa17}, the above statement on character triples does not depend on the choice of $(\e,\chi)\in\Omega(\overline{(\d,\vartheta)})$ nor on the representative $(\d,\vartheta)$ of $\overline{(\d,\vartheta)}$. We will make use of this fact without further reference.

As shown in the following lemma, it is no loss of generality to assume $\O_p(G)\leq \z(G)$ and consider $p$-chains with initial term $\O_p(G)$. This is an adaptation of a well known result (see, for instance, \cite[Theorem 9.16]{Nav18}).

\begin{lem}
\label{lem:Knorr-Robinson for Character Triple Conjecture}
Conjecture \ref{conj:Character Triple Conjecture} holds whenever $Z<\O_p(G)$.
\end{lem}

\begin{proof}
Consider $\d\in\mathfrak{N}(G,Z)$ with $\d=\{D_0<D_1\dots<D_n\}$. If $\O_p(G)\nleq D_n$, then define $\d^*$ to be the $p$-chain obtained by adding $\O_p(G)D_n$ to $\d$. Assume $\O_p(G)\leq D_n$ and let $k$ be the unique nonnegative integer such that $\O_p(G)\leq D_k$ and $\O_p(G)\nleq D_{k-1}$. If $\O_p(G)D_{k-1}=D_k$, then we define $\d^*$ by deleting the term $D_k$ from $\d$. If $\O_p(G)D_{k-1}<D_k$, then we define $\d^*$ by adding the term $\O_p(G)D_{k-1}$ to $\d$. This defines a self-inverse $\n_A(Z)$-equivariant bijection $^*:\mathfrak{N}(G,Z)\to \mathfrak{N}(G,Z)$ such that $|\d|=|\d^*|\pm 1$. In particular $G_\d=G_{\d^*}$ and we define $\Omega(\overline{(\d,\vartheta)}):=\overline{(\d^*,\vartheta)}$, for every $(\d,\vartheta)\in\C^d(B,Z)_+$.
\end{proof}

\subsection{A consequence of the Harris--Kn\"orr theorem}

Next, we collect some consequences of the Harris--Kn\"orr theorem  that will be used in the sequel. The reader should notice that Corollary \ref{cor:Harris-Knorr} and Corollary \ref{cor:Harris-Knorr and p-chains} below can also be deduced from \cite[Theorem 4.1]{Tur17}. However, we present here an elementary argument.

\begin{lem}
\label{lem:Harris-Knorr}
Let $N\unlhd G$ and $P$ be a $p$-subgroup of $N$. Consider a block $b\in\Bl(N\mid P)$ and its Brauer first main correspondent $b'\in\Bl(\n_N(P)\mid P)$. Let $B'\in\Bl(\n_G(P))$ and set $B:=(B')^G$. Then $B'$ covers $b'$ if and only if $B$ covers $b$.
\end{lem}

\begin{proof}
The result follows from the proof of the Harris--Kn\"orr theorem \cite{Har-Kno85}.
\end{proof}

If $P$ is a $p$-group acting via automorphisms on a $p'$-group $N$, we denote by $f_P:\irr_P(N)\to\irr(\c_N(P))$ the $P$-Glauberman correspondence (see \cite[Chapter 13]{Isa76} and \cite[\S 2.3]{Nav18}).

\begin{cor}
\label{cor:Harris-Knorr}
Let $N$ be a normal $p'$-subgroup of $G$ and $P$ be a $p$-subgroup of $G$. Consider $\mu\in\irr_P(N)$ and set $\mu':=f_P(\mu)\in\irr(\c_N(P))$. If $B'\in\Bl(\n_G(P))$, then $B'$ covers $\bl(\mu')$ if and only if $(B')^{N\n_G(P)}$ covers $\bl(\mu)$. Moreover, if $\mu$ is $G$-invariant, then $B'$ covers $\bl(\mu')$ if and only if $(B')^{G}$ covers $\bl(\mu)$.
\end{cor}

\begin{proof}
Let $b'$ be the unique block of $\n_{NP}(P)$ that covers $\bl(\mu')$, $b$ the unique block of $NP$ that covers $\bl(\mu)$ (see \cite[Corollary 9.6]{Nav98}) and notice that $b$ and $b'$ are Brauer first main correspondents over $P$. Now, ${\Bl(N\n_G(P)\mid \bl(\mu))}={\Bl(N\n_G(P)\mid b)}$ and ${\Bl(\n_G(P)\mid \bl(\mu'))}={\Bl(\n_G(P)\mid b')}$ and applying Lemma \ref{lem:Harris-Knorr} it follows that $B'$ covers $\bl(\mu')$ if and only if $(B')^{N\n_G(P)}$ covers $\bl(\mu)$. Moreover, if $\mu$ is $G$-invariant, then $\bl(\mu)$ is covered by $(B')^{N\n_G(P)}$ if and only if it is covered by $(B')^G$.
\end{proof}

With the same argument, we obtain a version of the above corollary for normal $p$-chains. 

\begin{cor}
\label{cor:Harris-Knorr and p-chains}
Let $N$ be a normal $p'$-subgroup of $G$ and $\d$ be a normal $p$-chain of $G$ with last term $P$. Consider $\mu\in\irr_P(N)$ and set $\mu':=f_P(\mu)\in\irr(\c_N(P))$. If $B'\in\Bl(G_\d)$, then $B'$ covers $\bl(\mu')$ if and only if $(B')^{NG_\d}$ covers $\bl(\mu)$. Moreover, if $\mu$ is $G$-invariant, then $B'$ covers $\bl(\mu')$ if and only if $(B')^G$ covers $\bl(\mu)$.
\end{cor}

\begin{proof}
The proof of Corollary \ref{cor:Harris-Knorr} applies with minor changes.
\end{proof}

\subsection{Construction of $N$-block isomorphic character triples}

We prove some useful results that can be used to construct $N$-block isomorphic character triples. First, we give a version of \cite[Theorem 3.14]{Nav-Spa14I} for our situation. This proposition allows to obtain new $N$-block isomorphic character triples involving irreducibly induced characters. This is the case, for instance, when we apply the Fong--Reynolds correspondence or the Clifford correspondence. Before proving this result, we need an easy lemma.

\begin{lem}
\label{lem:Irreducible induction}
Let $N\unlhd G$ and $\vartheta\in\irr(N)$. If $\vartheta^G\in\irr(G)$, then $\c_G(N)\leq N$.
\end{lem}

\begin{proof}
Set $H:=N\c_G(N)$ and observe that $\psi:=\vartheta^H\in\irr(H)$.  Since $\vartheta$ is $H$-invariant we have $\psi_N=e\vartheta$ with $e=|H:N|$. However $e=[\psi_N,\vartheta]=[\psi,\psi]=1$ and therefore $\c_G(N)\leq N$.
\end{proof}

\begin{prop}
\label{prop:Irreducible induction and double relation block}
Let $N\unlhd G$ and $G_0\leq G$. For $i=1,2$, consider $H_i\leq G$ such that $G=NH_i$ and set $M_i:=N\cap H_i$, $H_{0,i}:=G_0\cap H_i$, $M_{0,i}:=G_0\cap M_i$ and $N_0:=G_0\cap N\unlhd G_0$. Suppose that $G=G_0N$, that $H_i=H_{0,i}M_i$ and that $\varphi_i:=(\varphi_{0,i})^{M_i}\in\irr(M_i)$, for some $\varphi_{0,i}\in\irr(M_{0,i})$. If $(H_{0,1},M_{0,1},\varphi_{0,1})\iso{N_0}(H_{0,2},M_{0,2},\varphi_{0,2})$, there exists a defect group $D_i\in \delta(\bl(\varphi_i))$ such that $\c_G(D_i)\leq H_i$, and induction $\ind_{J_{0,i}}^{J_i}:\irr(J_{0,i}\mid \varphi_{0,i})\to \irr(J_i\mid \varphi_i)$ defines a bijection for every $N\leq J\leq G$, where $J_i:=J\cap H_i$ and $J_{0,i}:=J\cap H_{0,i}$, then $(H_1,M_1,\varphi_1)\iso{N}(H_2,M_2,\varphi_2)$.
\end{prop}

\begin{proof}
Assume $(H_{0,1},M_{0,1},\varphi_{0,1})\iso{N_0}(H_{0,2},M_{0,2},\varphi_{0,2})$ via $(\mathcal{P}_{0,1},\mathcal{P}_{0,2})$ and let $\alpha_{0,i}$ be the factor set of $\mathcal{P}_{0,i}$. Consider the canonical isomorphisms $l_{0,i}:H_{0,i}/M_{0,i}\to G_0/N_0$ and $l_i:H_i/M_i\to G/N$ and set $i_0:=l_{0,2}^{-1}\circ l_{0,1}$ and $i=l_2^{-1}\circ l_1$. If $j:G/N\to G_0/N_0$ and $j_i:H_i/M_i\to H_{0,i}/M_{0,i}$ are the canonical isomorphisms, then we have a commutative diagram
\begin{center}
\begin{tikzcd}
H_1/M_1\arrow[r, "l_1"]\arrow[rr, bend left={21pt}, "i"]\arrow[d,swap, "j_1"] &G/N\arrow[d, "j"] & H_2/M_2\arrow[l, swap, "l_2"]\arrow[d, "j_2"]
\\
H_{0,1}/M_{0,1}\arrow[r, swap, "l_{0,1}"]\arrow[rr, swap, bend right={21pt}, "i_0"] &G_0/N_0 & H_{0,2}/M_{0,2}\arrow[l, "l_{0,2}"]
\end{tikzcd}
\end{center}
As in \cite[Theorem 3.14]{Nav-Spa14I}, we consider the projective representation $\mathcal{P}_i:=(\mathcal{P}_{0,i})^{H_i}\in {\proj(H_i\mid \alpha_i)}$ defined as follows: let $\{t_{i,1},\dots, t_{i,n}\}$ be a $H_i$-transversal for $H_{0,i}$ contained in $M_i$, where $n:={|G:G_0|}={|H_i:H_{0,i}|}$. For every $x\in H_i$, let  
\begin{align*}
\mathcal{P}_{i,j,k}(x):=\begin{cases}
\mathcal{P}_{0,i}(t_{i,j}^{-1}xt_{i,k}), &\text{ if }t_{i,j}^{-1}xt_{i,k}\in H_{0,i}
\\
0,&\text{ otherwise }	
\end{cases}
\end{align*}
and define
$$\mathcal{P}_i(x):=\begin{pmatrix}
\mathcal{P}_{i,1,1}(x) &\dots &\mathcal{P}_{i,1,n}(x)
\\
\vdots &&\vdots
\\
\mathcal{P}_{i,n,1}(x) &\dots &\mathcal{P}_{i,n,n}(x)
\end{pmatrix}.$$
Then, $\mathcal{P}_i$ is a projective representation of $H_i$ associated with $\varphi_i=\varphi_{0,i}^{M_i}$ with factor set $\alpha_i$ satisfying $\alpha_i(x,y)=\alpha_{0,i}(j_i(x),j_i(y))$, for all $x,y\in H_i/M_i$. Since
$$\alpha_{0,1}(j_1(x),j_1(y))=\alpha_{0,2}(i_0(j_1(x)),i_0(j_1(y))),$$
we conclude that $\alpha_1(x,y)=\alpha_2(i(x),i(y))$, for all $x,y\in H_1/M_1$.

We claim that $\c_{H_i}(M_i)\leq G_0$. In this case, since $\c_G(N)\leq \c_G(D_i)\leq H_i$, we deduce $\c_G(N)\leq \c_{G_0}(N_0)$. To prove the claim, fix $x\in \c_{H_i}(M_i)$, set $J_i:=\langle M_i,x \rangle$ and $J_{0,i}:=G_0\cap J_i$ and let $\varphi_{i,x}$ be an extension of $\varphi_i$ to $J_i$. Since $\ind_{J_{0,i}}^{J_i}:\irr(J_{0,i}\mid \varphi_{0,i})\to \irr(J_i\mid \varphi_i)$ is a bijection, we can find an irreducible character $\varphi_{0,i,x}\in\irr(J_{0,i}\mid \varphi_{0,i})$ such that $\varphi_{0,i,x}^{J_i}=\varphi_{i,x}$. By Lemma \ref{lem:Irreducible induction} we conclude that $x\in \c_{J_i}(J_{0,i})\leq J_{0,i}\leq G_0$. This proves the claim, hence $\c_G(N)\leq\c_{G_0}(N_0)$. Now, since $\mathcal{P}_{0,1,\c_{G_0}(N_0)}$ and $\mathcal{P}_{0,2,\c_{G_0}(N_0)}$ are associated with the same scalar function and $[t_{i,j},\c_G(N)]=1$ for every $i=1,2$ and $j=1,\dots, n$, then the same is true for $\mathcal{P}_{1,\c_G(N)}$ and $\mathcal{P}_{2,\c_G(N)}$.

Next, fix $N\leq J\leq G$, set $J_0:=J\cap G_0$, $J_i:=J\cap H_i$ and $J_{0,i}:=J\cap H_{0,i}$, and consider the bijections given by the character triple isomorphisms induced by $(\mathcal{P}_{0,1},\mathcal{P}_{0,2})$ and $(\mathcal{P}_1,\mathcal{P}_2)$:
\begin{align*}
\sigma_{0,J_{0,1}}:\irr(J_{0,1}\mid \varphi_{0,1}) &\to \irr(J_{0,2}\mid \varphi_{0,2})
\\
\tr(\mathcal{Q}_{0,J_{0,1}}\otimes \mathcal{P}_{0,1,J_{0,1}}) &\mapsto \tr(\mathcal{Q}_{0,J_{0,2}}\otimes \mathcal{P}_{0,2,J_{0,2}})
\end{align*}
where $\mathcal{Q}_0\in\proj(J_0/N_0)$, and
\begin{align*}
\sigma_{J_1}:\irr(J_1\mid \varphi_1) &\to \irr(J_2\mid \varphi_2)
\\
\tr(\mathcal{Q}_{J_1}\otimes \mathcal{P}_{1,J_1}) &\mapsto \tr(\mathcal{Q}_{J_2}\otimes \mathcal{P}_{2,J_2})
\end{align*}
where $Q\in\proj(J/N)$. Observe that $\sigma_{J_1}\left(\psi_0^{J_1}\right)=\left(\sigma_{0,J_{0,1}}(\psi_0)\right)^{J_2}$, for all $\psi_0\in \irr(J_{0,1}\mid \varphi_{0,1})$. Let $\psi\in\irr(J_1\mid\varphi_1)$ and write $\psi=\psi_0^{J_1}$, for some $\psi_0\in\irr(J_{0,1}\mid \varphi_{0,1})$. Since by hypothesis $\bl(\psi_0)^{J_0}=\bl(\sigma_{0,J_{0,1}}(\psi_0))^{J_0}$, we conclude that $\bl(\psi)^J=\bl(\sigma_{J_1}(\psi))^J$. 
\end{proof}

Whenever we have a pair of $N$-isomorphic character triples, there is an induced strong isomorphism of character triples with some special properties. In the following lemma we describe some of these special features.

\begin{lem}
\label{lem:Bijections induced by triple isomorphisms}
Let $(H_1,M_1,\vartheta_1)\iso{N}(H_2,M_2,\vartheta_2)$ given by $(\mathcal{P}_1,\mathcal{P}_2)$ and, for $N\leq J\leq G=NH_i$, consider the bijection $\sigma_{J_1}^{(\mathcal{P}_1,\mathcal{P}_2)}:\irr(J_1\mid \vartheta_1)\to\irr(J_2\mid \vartheta_2)$, where $J_i:=J\cap H_i$. Let $\psi_1\in\irr(J_1\mid \vartheta_1)$ and $\psi_2:=\sigma_{J_1}^{(\mathcal{P}_1,\mathcal{P}_2)}(\psi_1)$. Then:
\begin{enumerate}
\item there exists $\wh{\mathcal{Q}}\in\proj(JH_{i,J,\psi_i}/N)$ such that $\wh{\mathcal{Q}}_{i,J_i}\otimes \mathcal{P}_{i,J_i}$ affords $\psi_i$, where $\wh{\mathcal{Q}}_i:=\mathcal{Q}_{H_{i,J,\psi_i}}$; 
\item if $\mathcal{D}_i:=\wh{\mathcal{Q}}_i\otimes \mathcal{P}_{i}$, then
$$(H_{1,J,\psi_1},J_1,\psi_1)\iso{J}(H_{2,J,\psi_2},J_2,\psi_2)$$
via $(\mathcal{D}_1,\mathcal{D}_2)$. Moreover
$$\sigma_{K_1}^{(\mathcal{D}_1,\mathcal{D}_2)}(\chi_1)=\sigma_{K_1}^{(\mathcal{P}_1,\mathcal{P}_2)}(\chi_1),$$
for every $J\leq K\leq JH_{i,J,\psi_i}$ and $\chi\in\irr(K_1\mid \psi_1)\subseteq \irr(K_1\mid \vartheta_1)$, where $K_i:=K\cap H_i$;
\item $d(\psi_1)-d(\psi_2)=d(\vartheta_1)-d(\vartheta_2)$.
\end{enumerate}
\end{lem}

\begin{proof}
First, as $JH_{1,J}=G_J=JH_{2,J}$, we may assume $J\unlhd G$. Moreover, since $(i,\sigma)$ is a strong isomorphism of character triples, we know that 
$$\sigma_{J_1}(\psi_1)^{x_2}=\sigma_{J_1^{x_1}}(\psi_1^{x_1})=\sigma_{J_1}(\psi_1^{x_1}),$$ 
for every $x_1\in H_1$ and $x_2\in H_2$ such that $i(M_1x_1)=M_2x_2$. In particular $i(H_{1,\psi_1}/M_1)=H_{2,\psi_2}/M_2$ and so $JH_{1,\psi_1}=JH_{2,\psi_2}$. Therefore, we may assume $H_i=H_{i,\psi_i}$.

By \cite[Theorem 3.3]{Spa17}, there exists $\mathcal{Q}\in\proj(J/N\mid \alpha^{-1}_{J\times J})$ such that $\psi_i$ is afforded by $\mathcal{Q}_{J_i}\otimes \mathcal{P}_{i,J_i}$. By \cite[Theorem 5.5]{Nav18}, there exists $\mathcal{D}_1\in\proj(H_1)$ such that $\mathcal{D}_{1,J_1}=\mathcal{Q}_{J_1}\otimes \mathcal{P}_{1,J_1}$. Arguing as in \cite[p. 707]{Nav-Spa14I}, relying on the proof of \cite[Theorem 8.16]{Nav98} we can find $\wh{\mathcal{Q}}_1\in\proj(H_1)$ such that
$$\mathcal{D}_1=\wh{\mathcal{Q}}_1\otimes \mathcal{P}_1$$
and $\wh{\mathcal{Q}}_{1,J_1}=\mathcal{Q}_{J_1}$. Since $N\leq \ker(\mathcal{Q})$, we deduce that $M_1\leq \ker(\mathcal{Q}_{J_1})\leq \ker(\wh{\mathcal{Q}}_1)$. Now $\wh{\mathcal{Q}}_1\in\proj(H_1/M_1)$ and, using the isomorphism $H_1/M_1\simeq G/N\simeq H_2/M_2$, we define $\wh{\mathcal{Q}}\in\proj(G/N)$ and $\wh{\mathcal{Q}}_2\in \proj(H_2/M_2)$. This proves (i). Set
$$\mathcal{D}_2:=\wh{\mathcal{Q}}_2\otimes \mathcal{P}_2.$$
We claim that $(H_1,J_1,\psi_1)\iso{J}(H_2,J_2,\psi_2)$ via $(\mathcal{D}_1,\mathcal{D}_2)$. Clearly the condition on the factor sets is satisfied. Moreover, since $\psi_i$ lies over $\vartheta_i$, we can find $D_i\in\delta(\bl(\psi_i))$ and $Q_i\in\delta(\bl(\vartheta_i))$ with $Q_i\leq D_i$ and $\c_G(Q_i)\leq H_i$. It follows that $\c_G(D_i)\leq H_i$. To conclude, we need to check the condition on block induction for $\sigma^{(\mathcal{D}_1,\mathcal{D}_2)}$. It's enough to show that $\sigma_{K_1}^{(\mathcal{D}_1,\mathcal{D}_2)}$ coincides with $\sigma_{K_1}^{(\mathcal{P}_1,\mathcal{P}_2)}$ on $\irr(K_1\mid \psi_1)$, for every $J\leq K\leq G$, where $K_i:=K\cap H_i$. Consider $\chi_i\in\irr(K_1\mid \psi_1)$ and let $\mathcal{R}\in\proj(K/J)$ such that $\chi_1=\tr(\mathcal{R}_{K_1}\otimes \mathcal{D}_{1,K_1})$. Then
\begin{align*}
\sigma_{K_1}^{(\mathcal{D}_1,\mathcal{D}_2)}(\chi_1)&=\tr(\mathcal{R}_{K_2}\otimes \mathcal{D}_{2,K_2})
\\
&=\tr(\mathcal{R}_{K_2}\otimes \wh{Q}_{K_2}\otimes \mathcal{P}_{2,K_2})
\\
&=\sigma_{K_1}^{(\mathcal{P}_1,\mathcal{P}_2)}(\tr(\mathcal{R}_{K_1}\otimes \wh{\mathcal{Q}}_{K_1}\otimes \mathcal{P}_{1,K_1}))
\\
&=\sigma_{K_1}^{(\mathcal{P}_1,\mathcal{P}_2)}(\tr\left(\mathcal{R}_{K_1}\otimes \mathcal{D}_{1,K_1}\right))
\\
&=\sigma_{K_1}^{(\mathcal{P}_1,\mathcal{P}_2)}(\chi_1).
\end{align*}
and the proof of (ii) is complete. To conclude, since $\psi_1(1)/\vartheta_1(1)=\psi_2(1)/\vartheta_2(1)$ by \cite[Lemma 11.24]{Isa76} and $|J:J_i|=|N:M_i|$, it follows that
$$p^{d(\psi_1)-d(\psi_2)}=\dfrac{|J_1|_p\psi_2(1)_p}{|J_2|_p\psi_1(1)_p}=\dfrac{|M_1|_p\vartheta_2(1)_p}{|M_2|_p\vartheta_1(1)_p}=p^{d(\vartheta_1)-d(\vartheta_2)}.$$
This finishes the proof.
\end{proof}

Given a defect preserving equivariant bijection respecting $N$-block isomorphic character triples, we show how to obtain another bijection over the given one that satisfies similar properties. For $N\unlhd G$ and $\mathcal{S}\subseteq \irr(N)$, we denote by $\irr(G\mid \mathcal{S})$ the set of $\chi\in\irr(G)$ that lies over some $\vartheta\in\mathcal{S}$.

\begin{prop}
\label{prop:Constructing bijections over bijections}
Let $K\unlhd A$, $A_0\leq A$ with $A=KA_0$ and, for every subgroup $X\leq A$, set $X_0:=X\cap A_0$. Consider $A_0$-stable subsets of characters $\mathcal{S}\subseteq \irr(K)$ and $\mathcal{S}_0\subseteq \irr(K_0)$. Assume there exists an $A_0$-equivariant bijection
$$\Psi:\mathcal{S}\to\mathcal{S}_0$$
such that
$$\left(A_\vartheta,K,\vartheta\right)\iso{K}\left(A_{0,\vartheta},K_0,\Psi(\vartheta)\right)$$
and 
$$\c_A(D)\leq A_0,$$
for every $\vartheta\in\mathcal{S}$ and some defect group $D$ of $\bl(\Psi(\vartheta))$. Then, for every $K\leq J\leq A$, there exists an $A_{0,J}$-equivariant bijection
$$\Phi_J:\irr(J\mid \mathcal{S})\to\irr(J_0\mid \mathcal{S}_0)$$
such that
$$\left(A_{J,\chi},J,\chi\right)\iso{J}\left(A_{0,J,\chi},J_0, \Phi_J(\chi)\right)$$
and
$$\c_A(Q)\leq A_0,$$
for every $\chi\in\irr(J\mid \mathcal{S})$ and some defect group $Q$ of $\bl(\Phi_J(\chi))$. Moreover $\Psi$ preserves the defect of characters if and only if so does $\Phi_J$.
\end{prop}

\begin{proof}
Consider an $\n_{A_0}(J)$-transversal $\mathbb{S}$ in $\mathcal{S}$ and define $\mathbb{S}_0:=\{\Psi(\vartheta)\mid \vartheta\in\mathbb{S}\}$. Since $\Psi$ is $A_0$-equivariant, it follows that $\mathbb{S}_0$ is an $\n_{A_0}(J)$-transversal in $\mathcal{S}_0$. For every $\vartheta\in\mathbb{S}$, with $\vartheta_0:=\Psi(\vartheta)\in\mathbb{S}_0$, we fix a pair of projective representations $(\mathcal{P}^{(\vartheta)},\mathcal{P}^{(\vartheta_0)}_0)$ giving $(A_\vartheta,K,\vartheta)\iso{K}(A_{0,\vartheta},K_0,\vartheta_0)$. Now, let $\mathbb{T}$ be an $\n_{A_0}(J)$-transversal in $\irr(J\mid \mathcal{S})$ such that every character $\chi\in\mathbb{T}$ lies above a character $\vartheta\in\mathbb{S}$ (this can be done by the choice of $\mathbb{S}$). Moreover, as $A=KA_0$, we have $J=KJ_0$ and therefore every $\chi\in\mathbb{T}$ lies over a unique $\vartheta\in\mathbb{S}$ by Clifford's theorem.

For $\chi\in\mathbb{T}$ lying over $\vartheta\in\mathbb{S}$, let $\varphi\in\irr(J_\vartheta\mid \vartheta)$ be the Clifford correspondent of $\chi$ over $\vartheta$. Set $\vartheta_0:=\Psi(\vartheta)\in\mathbb{S}_0$ and consider the $\n_{A_0}(J)_\vartheta$-equivariant bijection $\sigma_{J_\vartheta}:\irr(J_\vartheta\mid \vartheta)\to\irr(J_{0,\vartheta}\mid \vartheta_0)$ induced by our choice of projective representations $(\mathcal{P}^{(\vartheta)},\mathcal{P}_0^{(\vartheta_0)})$. Let $\varphi_0:=\sigma_{J_\vartheta}(\varphi)$. Since $\Psi$ is $A_0$-equivariant, we deduce that $J_{0,\vartheta}=J_{0,\vartheta_0}$ and therefore $\Phi_J(\chi):=\varphi^{J_0}$ is irreducible by the Clifford correspondence. Then, we define
$$\Phi_J\left(\chi^x\right):=\Phi_J(\chi)^x,$$ 
for every $\chi\in\mathbb{T}$ and $x\in \n_{A_0}(J)$. This defines an $\n_{A_0}(J)$-equivariant bijection $\Psi:\irr(J\mid \mathcal{S})\to\irr(J_0\mid \mathcal{S}_0)$. Furthermore, using Lemma \ref{lem:Bijections induced by triple isomorphisms} it's clear that $\Psi$ preserves the defect of characters if and only if so does $\Phi_J$.

Next, using the fact that $\left(A_\vartheta,K,\vartheta\right)\iso{K}\left(A_{0,\vartheta},K_0,\vartheta_0\right)$ together with Lemma \ref{lem:Bijections induced by triple isomorphisms}, we have
$$\left(A_{\vartheta,J_\vartheta,\psi},J_\vartheta,\psi\right)\iso{J_\vartheta}\left(A_{0,\vartheta,J_\vartheta,\psi},J_{0,\vartheta},\psi_0\right)$$
and, because $A_{\vartheta,J}\leq A_{\vartheta,J_{\vartheta}}$, we obtain
\begin{equation}
\label{eq:Constructing bijections over bijections 1}
\left(A_{\vartheta,J,\psi},J_\vartheta,\psi\right)\iso{J_\vartheta}\left(A_{0,\vartheta,J,\psi},J_{0,\vartheta},\psi_0\right).
\end{equation}
By hypothesis there exists a defect group $D$ of $\bl(\vartheta_0)$ such that $\c_A(D)\leq A_0$. Since $\bl(\chi_0)$ covers $\bl(\vartheta_0)$ we can find a defect group $Q$ of $\bl(\chi_0)$ such that $D\leq Q$. It follows that $\c_A(Q)\leq \c_A(D)\leq A_0$. Finally, we obtain
$$\left(A_{J,\chi},J,\chi\right)\iso{J}\left(A_{0,J,\chi},J_0, \Phi_J(\chi)\right)$$
by applying Proposition \ref{prop:Irreducible induction and double relation block} together with \eqref{eq:Constructing bijections over bijections 1}.
\end{proof}

We end this section with an elementary but useful observation. Suppose to have $N$-block isomorphic character triples and that $N\leq \wh{N}$. Under certain assumptions, it's possible to deduce that those character triples are in fact $\wh{N}$-block isomorphic.

\begin{lem}
\label{lem:Relation on triple wrt different groups}
Let $(H_1,M_1,\vartheta_1)\iso{N}(H_2,M_2,\vartheta_2)$ with $H_iN=G$. Suppose that $G\leq \wh{G}$ and let $N\leq \wh{N}\unlhd \wh{G}$ with $\wh{G}=G\wh{N}$ and $N=G\cap \wh{N}$. If $\c_{\wh{G}}(D_i)\leq G$ for some $D_i\in\delta(\bl(\vartheta_i))$, then $(H_1,M_1,\vartheta_1)\iso{\wh{N}}(H_2,M_2,\vartheta_2)$.
\end{lem}

\begin{proof}
This follows directly from Definition \ref{def:N-block isomorphic character triples}.
\end{proof}

\section{$N$-block isomorphic character triples and Glauberman correspondence}

The aim of this section is to prove Theorem \ref{thm:Really above the Glauberman correspondence with block relation with chains} which will be one of the main ingredients in the final proof. To prove this result, we need to extend the bijection given in \cite[Theorem 5.13]{Nav-Spa14I} to characters of positive height. This is done in Proposition \ref{prop:Above the Glauberman correspondence with block relation} for the case where the $D$-correspondence coincides with the Glauberman correspondence. Moreover, we obtain a canonical bijection.  

Let $N\unlhd G$ and $\vartheta\in\irr_G(N)$ such that $(o(\vartheta)\vartheta(1),|G:N|)=1$. We denote by $\vartheta^\diamond$ the canonical extension of $\vartheta$ to $G$, i.e. the unique extension of $\vartheta$ to $G$ such that $(o(\vartheta^\diamond),|G:N|)=1$ (see \cite[Corollary 8.16]{Isa76}). To prove Proposition \ref{prop:Above the Glauberman correspondence with block relation}, in addition to the argument developed in \cite[\S 5]{Nav-Spa14I}, we need the following result on the extendibility of the canonical extension. In what follows we will often use the following easy fact: if $H\leq G$ and $\chi\in\irr(G)$ such that $\chi_H\in\irr(H)$, then $o(\chi_H)$ divides $o(\chi)$.

\begin{lem}
\label{lem:Canonical extension}
Let $N$ be a normal $p'$-subgroup of $G$ and $P$ a $p$-subgroup of $G$ such that $K:=NP\unlhd G$. Consider $\mu\in\irr_G(N)$ and let $\mu^\diamond\in\irr_G(K)$. Then $\mu$ extends to $G$ if and only if $\mu^\diamond$ extends to $G$.
\end{lem}

\begin{proof}
One implication is trivial. Notice that $\mu^\diamond$ is $G$-invariant since $\mu$ is $G$-invariant. Assume that $\mu$ has an extension $\chi\in\irr(G)$. We have to show that $\mu^\diamond$ extends to $H$, for every $H/K\in\syl_q(G/K)$ and every prime $q$. If $q=p$, then $\mu$ has a canonical extension to $H$, which is also an extension of $\mu^\diamond$.

Assume $q\neq p$ and consider $\lambda\in\irr(K/N)$ such that $\mu^\diamond=\lambda\chi_K$. Notice that, as $\mu^\diamond$ and $\chi_K$ are $G$-invariant, the character $\lambda$ is $G$-invariant. Since $K/N$ is a $p$-group and $H/K$ is a $q$-group, we deduce that $\lambda$ has a canonical extension $\lambda^\diamond$ to $H$. Then $\lambda^\diamond\chi_H$ is an extension of $\mu^\diamond$. This concludes the proof.
\end{proof}

\begin{hyp}
\label{hyp:Glauberman section}
Let $N$ be a normal $p'$-subgroup of $A$ and $P$ be a $p$-subgroup of $A$ such that $K:=NP\unlhd A$. Consider $\mu\in\irr_A(N)$ and its Glauberman correspondent $f_P(\mu)\in\irr_{A_P}(N_P)$. Let $\mu^\diamond\in\irr_A(K)$ and $f_P(\mu)^\diamond\in\irr_{A_P}(K_P)$ be the canonical extensions respectively of $\mu$ and $f_P(\mu)$.
\end{hyp}

Now, our aim is to obtain an adaptation of \cite[Proposition 5.12]{Nav-Spa14I} that includes canonical extensions (see Lemma \ref{lem:Canonical extension and values} below). This is done by proceeding as in \cite[\S 5]{Nav-Spa14I} and using Lemma \ref{lem:Canonical extension}.

\begin{lem}
\label{lem:Canonical extension and irreducible constituents}
Assume Hypothesis \ref{hyp:Glauberman section} and let $C$ be an abelian normal subgroup of $A$ with $C\leq \c_A(K)$. Suppose that $\mu^\diamond$ has an extension $\wt{\mu}$ to $A$. Then there exists an extension $\widetilde{f_P(\mu)}$ of $f_P(\mu)^\diamond$ to $A_P$ such that
$$\irr\left(\wt{\mu}_C\right)=\irr\left(\widetilde{f_P(\mu)}_C\right).$$
\end{lem}

\begin{proof}
Write $C_p:=\O_p(C)$ and $C_{p'}:=\O_{p'}(C)$ and set $\kappa:=\wt{\mu}_{NC_{p'}}$. Let $\kappa^\diamond$ be the canonical extension of $\kappa$ to $KC$. Since $\kappa$ extends to $A$, there exists an extension $\wt{\kappa}$ of $\kappa^\diamond$ to $A$ by Lemma \ref{lem:Canonical extension}. Observe that $\kappa^\diamond$ extends $\mu^\diamond$ and so does $\wt{\kappa}$. Now, by \cite[Corollary 6.17]{Isa76}, there exists a linear character $\eta\in\irr(A/K)$ such that $\wt{\mu}=\wt{\kappa}\eta$. Let $\lambda$ and $\lambda_1$ be the unique irreducible constituent respectively of $\wt{\mu}_C$ and of $\wt{\kappa}_C$. Then $\lambda=\lambda_1\eta_C$. Next, consider the $P$-Glauberman correspondent $f_P(\kappa)\in\irr((NC_{p'})_P)$ of $\kappa$ and let $f_P(\kappa)^\diamond$ be its canonical extension to $(KC)_P$. Using \cite[Theorem 6.5]{Tur08I} and \cite[Theorem 7.12]{Tur09}, as $\kappa$ extends to $A$, we conclude that $f_P(\kappa)$ extends to $A_P$. By Lemma \ref{lem:Canonical extension} there exists an extension $\wt{f_P(\kappa)}$ of $f_P(\kappa)^\diamond$ to $A_P$. As before, notice that $\wt{f_P(\kappa)}$ is an extension of $f_P(\mu)^\diamond$. Define $\wt{f_P(\mu)}:=\wt{f_P(\kappa)}\eta_{A_P}$. Since $K_P\leq \ker(\eta_{A_P})$, it follows that $\wt{f_P(\mu)}$ is an extension of $f_P(\mu)^\diamond$. If $\lambda'$ and $\lambda_1'$ are the unique irreducible constituents respectively of $\wt{f_P(\mu)}_C$ and $\wt{f_P(\kappa)}_C$, then $\lambda'=\lambda_1'\eta_C$. Therefore, in order to conclude, it is enough to show that $\lambda_1=\lambda_1'$. Write $\lambda_1=\lambda_{1,p}\times \lambda_{1,p'}$ and $\lambda_1'=\lambda_{1,p}'\times \lambda_{1,p'}'$, with $\lambda_{1,p},\lambda_{1,p}'\in\irr(C_p)$ and $\lambda_{1,p'},\lambda_{1,p'}'\in\irr(C_{p'})$. First, because $f_P(\kappa)$ is an irreducible constituent of $\kappa_{NC_{p'}}$ and $C_{p'}\leq \z(NC_{p'})$, it follows that
$$\irr\left(\wt{\kappa}_{C_{p'}}\right)=\irr\left(\kappa_{C_p'}\right)=\irr\left(f_P(\kappa)_{C_{p'}}\right)=\irr\left(\wt{f_P(\kappa)}_{C_{p'}}\right)$$
and therefore $\lambda_{1,p'}=\lambda_{1,p'}'$. Observe that $\wt{\kappa}_{N\times C_p}=(\kappa^\diamond)_{N\times C_p}=\mu\times \lambda_{1,p}$. Since $p$ does not divide $o(\kappa^\diamond)$, it follows that $p$ does not divide $o(\mu\times \lambda_p)$. In particular $(p,o(\lambda_p))=1$ and therefore $\lambda_{1,p}=1_{C_p}$. By the same argument, we obtain $\lambda_{1,p}'=1_{C_p}$. This shows that $\lambda_1=\lambda_1'$ and the proof is complete.
\end{proof}

Next, we extend Lemma \ref{lem:Canonical extension and irreducible constituents} to the case where $C$ is not necessarily abelian.

\begin{cor}
\label{cor:Canonical extension and irreducible constituents}
Assume Hypothesis \ref{hyp:Glauberman section} and suppose that $\mu^\diamond$ has an extension $\wt{\mu}$ to $A$. Then there exists an extension $\widetilde{f_P(\mu)}$ of $f_P(\mu)^\diamond$ to $A_P$ such that
$$\irr\left(\wt{\mu}_{\c_A(K)}\right)=\irr\left(\widetilde{f_P(\mu)}_{\c_A(K)}\right).$$
\end{cor}

\begin{proof}
Set $C:=\c_A(K)$, $C':=[C,C]$ and $\overline{A}:=A/C'$. Since $\wt{\mu}_K$ is irreducible, as remarked before \cite[Definition 2.7]{Spa18}, we have $C\leq \z(\wt{\mu})$ and \cite[Lemma 2.27]{Isa76} implies that $\wt{\mu}_C=\mu(1)\lambda$, for some linear character $\lambda\in\irr(C)$. In particular $C'\leq \ker(\lambda)\leq \ker(\wt{\mu})$. It follows that $C'\cap K$ is contained in $\ker(\mu^\diamond)$ and $\ker(f_P(\mu)^\diamond)$ while $C'\cap N$ is contained in $\ker(\mu)$ and $\ker(f_P(\mu))$. Via the canonical isomorphism $\overline{N}\simeq N/C'\cap N$, we can identify $\mu$ with a character $\overline{\mu}$ of $\overline{N}$. Similarly we can consider $\overline{\mu^\diamond}$ as a character of $\overline{K}$, $\overline{f_P(\mu)}$ as a character of $\overline{N_P}$ and $\overline{f_P(\mu)^\diamond}$ as a character of $\overline{K_P}$. Notice that $\overline{A_P}=\overline{A}_{\overline{P}}$, $\overline{K_P}=\overline{K}_{\overline{P}}$ and $\overline{N_P}=\overline{N}_{\overline{P}}$. By \cite[Lemma 5.10]{Nav-Spa14I} the character $\overline{f_P(\mu)}$ coincides with the $\overline{P}$-Glauberman correspondent $f_{\overline{P}}(\overline{\mu})$ of $\overline{\mu}$. Moreover $\overline{\mu^\diamond}$ and $\overline{f_P(\mu)^\diamond}$ are the canonical extensions of $\overline{\mu}$ and of $\overline{f_P(\mu)}$. Applying Lemma \ref{lem:Canonical extension and irreducible constituents}, we find an extension $\psi$ of $\overline{f_P(\mu)^\diamond}$ to $\overline{A}_{\overline{P}}$ such that $\irr\left(\overline{\wt{\mu}}_{\overline{C}}\right)=\irr\left(\psi_{\overline{C}}\right)$, where $\overline{\wt{\mu}}$ is the character of $\overline{A}$ corresponding to $\wt{\mu}$ via inflation. Now the inflation $\widetilde{f_P(\mu)}\in\irr(A_P)$ of $\psi$ satisfies the required hypothesis.
\end{proof}

Recall that, if $\mathbf{R}$ is the ring of algebraic integers and $\mathbf{S}$ is the localization of $\mathbf{R}$ at some maximal ideal containing $p\mathbf{R}$, then $^*:\mathbf{S}\to\mathbb{F}$ denotes the canonical epimorphism, where $\mathbb{F}$ is the residue field of characteristic $p$ (see \cite[Chapter 2]{Nav98} for details).

\begin{lem}
\label{lem:Canonical extension and values}
Assume Hypothesis \ref{hyp:Glauberman section}. If $\mu^\diamond$ extends to $\wt{\mu}\in\irr(A)$, then there exists an extension $\wt{f_P(\mu)}$ of $f_P(\mu)^\diamond$ to $A_P$ such that 
$$\irr\left(\wt{\mu}_{\c_A(K)}\right)=\irr\left(\wt{f_P(\mu)}_{\c_A(K)}\right)$$
and
$$\wt{\mu}(x)^*=e\wt{f_P(\mu)}(x)^*,$$
for every $p$-regular $x\in A$ with $P\in \syl_p(\c_K(x))$, where $e:=[\mu_{N_P},f_P(\mu)]$.
\end{lem}

\begin{proof}
By Lemma \ref{lem:Canonical extension and irreducible constituents} there exists an extension $\chi$ of $f_P(\mu)^\diamond$ that satisfies the first condition. In order to conclude, it is enough to find a linear character $\wt{\xi}\in\irr(A_P/\c_A(K)K_P)$ such that $\widetilde{f_P(\mu)}:=\wt{\xi}\cdot\chi$ satisfies the second condition.

First, we construct the linear character $\wt{\xi}$. Let $x$ be a $p$-regular element of $\c_A(P)K_P$, set $N^{(x)}:=N\langle x\rangle$, $K^{(x)}:=K\langle x \rangle$ and observe that $(N^{(x)})_P=(N_p)^{(x)}:=K_P\langle x\rangle$ and $(K^{(x)})_P=(K_p)^{(x)}:=K_P\langle x\rangle$. 
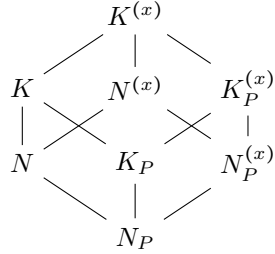
\begin{center}
\begin{tikzpicture}[x=1.5cm, y=1cm]
\node(N) at (0,0) {$N$};
\node(K) at (0,1) {$K$};
\node(Nx) at (1,1) {$N^{(x)}$};
\node(Kx) at (1,2) {$K^{(x)}$};
\node(NP) at (1,-1) {$N_P$};
\node(KP) at (1,0) {$K_P$};
\node(NPx) at (2,0) {$N_P^{(x)}$};
\node(KPx) at (2,1) {$K_P^{(x)}$};

\path[-]

(N) edge node {} (K)
(N) edge node {} (NP)
(N) edge node {} (Nx)
(K) edge node {} (KP)
(K) edge node {} (Kx)
(Nx) edge node {} (Kx)
(Nx) edge node {} (NPx)
(NP) edge node {} (KP)
(NP) edge node {} (NPx)
(KPx) edge node {} (Kx)
(KPx) edge node {} (KP)
(KPx) edge node {} (NPx);
\end{tikzpicture}
\end{center}
Since $x$ is $p$-regular, the subgroup $N^{(x)}$ has order coprime to $p$ and we can consider the $P$-Glauberman correspondent $f_P(\wt{\mu}_{N^{(x)}})$ of $\wt{\mu}_{N^{(x)}}$. Moreover $f_P(\wt{\mu}_{N^{(x)}})_{N_P}=f_P(\mu)$ by \cite[Theorem A]{Isa-Nav91}. Now, if $f_P(\wt{\mu}_{N^{(x)}})^\diamond$ is the canonical extension of $f_P(\wt{\mu}_{N^{(x)}})$ to $K_P^{(x)}$, then $(f_P(\wt{\mu}_{N^{(x)}})^\diamond)_{K_P}=f_P(\mu)^\diamond$. Since $\chi_{K_P^{(x)}}$ is another extension of $f_P(\mu)^\diamond$ to $K_P^{(x)}$, it follows that there exists a unique linear character $\xi^{(x)}\in\irr(K_P^{(x)}/K_P)$ such that $\xi^{(x)}\chi_{K_P^{(x)}}=f_P(\wt{\mu}_{N^{(x)}})^\diamond$. We define the map
\begin{align*}
\xi:\c_A(P)K_P&\to \mathbb{C}
\\
x &\mapsto \xi^{(x_{p'})}(x_{p'}).
\end{align*}
We claim that $\xi$ is a linear character of $\c_A(P)K_P$ with an extension $\wt{\xi}$ to $A_P$. To show that $\xi$ is an irreducible character we apply \cite[Corollary 8.12]{Isa76}. Clearly $\xi(1)=1$. Next, in order to show that $\xi$ is a class function we check that $\xi^{(x^n)}=(\xi^{(x)})^n$, for every $n\in A_P$ and every $p$-regular $x\in\c_A(P)K_P$. If this is the case, then
$$\xi(x^n)=\xi^{((x^n)_{p'})}((x^n)_{p'})=\xi^{((x_{p'})^n)}((x_{p'})^n)=(\xi^{(x_{p'})})^n((x_{p'})^n)=\xi^{(x_{p'})}(x_{p'})=\xi(x),$$
for every $x\in \c_A(P)K_P$ and $n\in A_P$. In particular $\xi$ is a class function. To prove the claim, just notice that $(\chi_{K_P^{(x)}})^n=\chi_{K_P^{(x^n)}}$ and that $(f_P(\wt{\mu}_{N^{(x)}})^\diamond)^n$ is the canonical extension of $f_P(\wt{\mu}_{N^{(x)}})^n=f_P(\wt{\mu}_{N^{(x^n)}})$, for every $n\in A_P$ and every $p$-regular $x\in\c_A(P)K_P$. Next, since $\xi^{(x)}=\xi^{(x^{-1})}$ for every $p$-regular $x\in\c_A(P)K_P$, we deduce that $\xi(x^{-1})=\xi^{-1}(x)$, for every $x\in \c_A(P)K_P$, and therefore $[\xi,\xi]=1$. Finally, fix $S\times T\leq \c_A(P)K_P$ with $S$ a $p$-group and $T$ a $p'$-group. Observe that $\xi_S=1_S$. On the other hand $\chi_{K_PT}$ and $f_P(\wt{\mu}_{NT})^\diamond$ are both extensions of $f_P(\mu)^\diamond$ and we can find a linear character $\lambda\in\irr(K_PT/K_P)$ such that $\lambda\chi_{K_PT}=f_P(\wt{\mu}_{NT})^\diamond$. Moreover, for every $x\in T$, we have $\left(f_P(\wt{\mu}_{NT})^\diamond\right)_{K_P^{(x)}}=f_P(\wt{\mu}_{N^{(x)}})^\diamond$ and therefore $\xi_T=\lambda_T$. It follows that $\xi_{S\times T}\in\mathbb{Z}\irr(S\times T)$ and hence $\xi$ is a linear character by \cite[Corollary 8.12]{Isa76}.

Next, we show that $\xi$ extends to $A_P$. To do so, we use \cite[Theorem 6.26]{Isa76}. Let $q$ be a prime dividing $o(\xi)$ and consider $S_q/\c_A(P)K_P\in\syl_q(A_P/\c_A(P)K_P)$. Noticing that every $p$-element $x$ of $\c_A(P)K_P$ is contained in $\ker(\xi)$, we deduce that $p$ does not divide $|\c_A(P)K_P:\ker(\xi)|$ and hence $q\neq p$. Let $Q\in\syl_q(A_P/N_P)$ such that $S_q=\c_A(P)K_PQ$ and define $Q_1:=Q\cap \c_A(P)K_P$ and $\xi_1:=\xi_{Q_1}$. By \cite[Lemma 4.1]{Spa10}, we deduce that $\xi$ extends to $A_P$ if and only if $\xi_{Q_1}$ extends to $Q$. We are going to check the latter condition. Because $Q_1\leq A_P$ we deduce that $NQ_1$ is a $P$-invariant $p'$-group and that $(NQ_1)_P=N_PQ_1=Q_1$. We also have $KQ_1=(NQ_1)\rtimes P$ and $(KQ_1)_P=K_PQ_1=Q_1P$. Now we can consider the $P$-Glauberman correspondent $f_P(\wt{\mu}_{NQ_1})$ and its canonical extension $f_P(\wt{\mu}_{NQ_1})^\diamond$ to $Q_1P$. By \cite[Theorem A]{Isa-Nav91} we have $f_P(\wt{\mu}_{NQ_1})_{N_P}=f_P(\mu)$ and so $(f_P(\wt{\mu}_{NQ_1})^\diamond)_{K_P}=f_P(\mu)^\diamond$. Using Lemma \ref{lem:Canonical extension}, we obtain an extension $\psi$ of $f_P(\wt{\mu}_{NQ_1})^\diamond$ to $(KQ)_P=K_PQ$. By Gallagher's theorem there exists a unique linear character $\nu\in\irr(K_PQ/K_P)$ such that $\chi_{K_PQ}\cdot\nu=\psi$. Finally, for every $x\in Q_1$, we have
\begin{align*}
\xi^{(x)}\chi_{K_P^{(x)}}&=f_P(\wt{\mu}_{N^{(x)}})^\diamond=(f_P(\wt{\mu}_{NQ_1})_{N_P^{(x)}})^\diamond
\\
&=(f_P(\wt{\mu}_{NQ_1})^\diamond)_{K_P^{(x)}}=(\psi_{PQ_1})_{K_P^{(x)}}
\\
&=\psi_{K_P^{(x)}}=\chi_{K_P^{(x)}}\nu_{K_P^{(x)}}
\end{align*}
and it follows that $\nu_{Q_1}=\xi_1$. This shows that $\nu_Q$ is an extension of $\xi_1$ to $Q$ and therefore $\xi$ extends to $S_q$. We conclude that $\xi$ has an extension $\wt{\xi}$ to $A_P$.

Define $\widetilde{f_P(\mu)}:=\wt{\xi}\chi$. By \cite[Theorem 2.6]{Nav-Spa14II} we deduce that
$$\wt{\mu}(x)^*=\wt{\mu}_{N^{(x)}}(x)^*=ef_P(\wt{\mu}_{N^{(x)}})(x)^*=e(\xi(x)\chi(x))^*=e\widetilde{f_P(\mu)}(x)^*,$$
for every $p$-regular $x\in E$ such that $P\in\syl_p(\c_K(x))$. 

It remains to show that $C:=\c_A(K)$ is contained in the kernel of $\xi$. First, observe that $\ker(\xi)$ contains $C':=[C,C]\leq \ker(\xi_C)$. Moreover, $\ker(\xi)$ contains every $p$-element of $C$. Since $C/C'$ is abelian it's enough to show that every $p$-regular element $x$ of $C$ lies in $\ker(\xi)$. By the Alperin argument we know that $B:=\bl(\wt{\mu}_{K^{(x)}})$ and $B':=\bl(f_P(\wt{\mu}_{N^{(x)}})^\diamond)$ are Brauer correspondents with $B$ covering $b:=\bl(\wt{\mu}_{N^{(x)}})=\{\wt{\mu}_{N^{(x)}}\}$ and $B'$ covering $b':=\bl(f_P(\wt{\mu}_{N^{(x)}}))=\{f_P(\wt{\mu}_{N^{(x)}})\}$. According to \cite[Theorem 4.14]{Nav98} it follows that $\lambda_B=\lambda_{B'}\circ \Br_P$. Since $x\in\c_A(K)$, we have $x\leq \z(K^{(x)})$ and hence
$$\lambda_B(x)=\lambda_B\left(\left(x^{K^{(x)}}\right)^+\right)=\lambda_{B'}\left(\left(x^{K^{(x)}}\cap \c_{K^{(x)}}(P)\right)^+\right)=\lambda_{B'}(x).$$ 
By \cite[Theorem 9.5]{Nav98}, we conclude that
$$\left(\frac{\wt{\mu}(x)}{\wt{\mu}(1)}\right)^*=\lambda_B(x)=\lambda_{B'}(x)=\left(\frac{f_P(\wt{\mu}_{N^{(x)}})(x)}{f_P(\wt{\mu}_{N^{(x)}})(1)}\right)^*=\left(\frac{\widetilde{f_P(\mu)}(x)}{\widetilde{f_P(\mu)}(1)}\right)^*.$$
As $\irr(\wt{\mu}_C)=\irr(\chi_C)$ and $x$ is $p$-regular, we obtain
$$\dfrac{\chi(x)}{\chi(1)}=\dfrac{\wt{\mu}(x)}{\wt{\mu}(1)}=\dfrac{\widetilde{f_P(\mu)}(x)}{\widetilde{f_P(\mu)}(1)}$$
and, in particular, $\xi(x)=1$. This concludes the proof.
\end{proof}

The following result extends the bijection given in \cite[Theorem 5.13]{Nav-Spa14I} (in the case where $K$ of \cite[Theorem 5.13]{Nav-Spa14I} is a $p'$-group) to characters of positive height. In this particular situation we obtain a canonical bijection.

\begin{prop}
\label{prop:Above the Glauberman correspondence with block relation}
Assume Hypothesis \ref{hyp:Glauberman section}. Then there exists a canonical $A_P$-equivariant bijection
\begin{align*}
\Psi_{\mu,P}:\irr(K\mid \mu)&\to\irr(K_P\mid f_P(\mu))
\\
\mu^\diamond\nu&\mapsto f_P(\mu)^\diamond\nu_{K_P},
\end{align*}
for every $\nu\in \irr(K/N)$. Moreover $\Psi_{\mu,P}$ preserves the defect of characters and
$$(A_\vartheta, K, \vartheta)\iso{K}(A_{P,\vartheta}, K_P, \Psi_{\mu,P}(\vartheta))$$
and
$$\c_{A}(D)\leq A_P,$$
for every $\vartheta\in\irr(K\mid \mu)$ and some defect group $D$ of $\hspace{2pt}\bl(\Psi_{\mu,P}(\vartheta))$.
\end{prop}

\begin{proof}
Since $K=N\rtimes P$ and $K_P=N_P\times P$ are $p$-nilpotent groups with Sylow $p$-subgroup $P$, $\mu$ is $K$-invariant and $f_P(\mu)$ is $K_P$-invariant, we have $\irr(K\mid \mu)=\irr(\bl(\mu^\diamond))=\{\mu^\diamond\nu\mid \nu\in\irr(K/N)\}$ and $\irr(K_P\mid f_P(\mu))=\irr(\bl(f_P(\mu)^\diamond))=\{f_P(\mu)^\diamond\nu\mid \nu\in\irr(K_P/N_P)\}$. Thus, we obtain a defect preserving $A_P$-quivariant bijection by setting
$$\Psi_{\mu,P}(\mu^\diamond\nu):=f_P(\mu)^\diamond\nu_{K_P},$$
for every $\nu\in\irr(K/N)$.
Furthermore, as $P$ is a common defect group of the two blocks $\bl(\vartheta)$ and $\bl(\Psi_{\mu,P}(\vartheta))$, the condition on defect groups is satisfied.

Consider $\mathcal{P}\in\proj(A\mid \alpha)$ a projective representation associated with $\mu^\diamond$ and observe that $\mathcal{P}$ is also associated with $\mu$. Let $\wh{A}$ be the central extension of $A$ defined by $\mathcal{P}$ and $\epsilon:\wh{A}\to A$ be the map given by $\epsilon(x,s):=x$, for every $x\in A$ and $s\in S$, with kernel $S:=\langle\alpha(x,y)\mid x,y\in A\rangle$. For $H\leq A$, set $\wh{H}:=\epsilon^{-1}(H)$. Since $\alpha$ is constant on $K$-cosets and $\alpha(1,1)=1$, the set $H_0:=\{(h,1)\mid h\in H\}$ is a subgroup of $\wh{A}$, whenever $H\leq K$. In this case let $\vartheta_0\in\irr(H_0)$ be the character corresponding to $\vartheta\in\irr(H)$ via the isomorphism $\epsilon_{H_0}:H_0\to H$. Moreover $\wh{H}=H_0\times S$ and we define $\wh{\vartheta}:=\vartheta_0\times 1_S\in\irr(\wh{H})$. Notice that $(\mu^\diamond)_0\in\irr(K_0)$ is the canonical extension of $\mu_0$ and that $\wh{\mu^\diamond}\in\irr(\wh{K})$ is the canonical extension of $\wh{\mu}$. Furthermore $f_P(\mu)_0=f_{P_0}(\mu_0)$ and $(f_P(\mu)^\diamond)_0$ is its canonical extension. As no confusion arise, we just write $\mu^\diamond_0$ (resp. $f_P(\mu)^\diamond_0$) instead of $(\mu^\diamond)_0=(\mu_0)^\diamond$ (resp. $(f_P(\mu)^\diamond)_0=f_{P_0}(\mu_0)^\diamond$).

Recall that the map defined by $\wh{\mathcal{P}}(x,s):=s\mathcal{P}(x)$, for all $(x,s)\in \wh{A}$, is an irreducible representation of $\wh{A}$ affording an extension $\tau$ of $\mu^\diamond_0$. Set $S_p:=\O_p(S)$, $S_{p'}:=\O_{p'}(S)$, $M:=N_0\times S_{p'}$, $Q:=P_0\times S_p$, and notice that $\wh{K}=M\rtimes Q$, $M_Q=(N_P)_0\times S_{p'}$ and $\wh{K}_Q=\wh{K_P}$. Let $\varphi:=\tau_M\in\irr_{\wh{A}}(M)$ and consider its canonical extension $\varphi^\diamond\in\irr(\wh{K})$. By Lemma \ref{lem:Canonical extension}, there exists an extension $\wt{\varphi}$ of $\varphi^\diamond$ to $\wh{A}$. Since $\wt{\varphi}_{K_0}$ is an extension of $\mu_0$ with $o(\wt{\varphi}_{K_0})$ dividing $o(\varphi^\diamond)$, we deduce that $\wt{\varphi}_{K_0}=\mu^\diamond_0$. Now, if $\wh{\mathcal{R}}$ is an irreducible representation of $\wh{A}$ affording $\wt{\varphi}$, then $\mathcal{R}(x):=\wh{\mathcal{R}}(x,1)$ defines a projective representation of $A$ associated with $\mu^\diamond$. Replacing $\mathcal{P}$ with $\mathcal{R}$, we may assume that $\tau$ extends $\varphi^\diamond$.

Now, Lemma \ref{lem:Canonical extension and values} yields an extension $\wt{f_Q(\varphi)}$ of $f_Q(\varphi)^\diamond$ to $\wh{A}_Q=\wh{A_P}$ such that 
\begin{equation}
\label{eq:Canonical bijection consituents}
\irr\left(\wt{\varphi}_{\c_{\wh{A}}(\wh{K})}\right)=\irr\left(\wt{f_Q(\varphi)}_{\c_{\wh{A}}(\wh{K})}\right)
\end{equation}
and
\begin{equation}
\label{eq:Canonical bijection values}
\wt{\varphi}(x)^*=e\wt{f_Q(\varphi)}(x)^*,
\end{equation}
for every $p$-regular $x\in \wh{A}$ such that $Q\in\syl_p(\c_{\wh{K}}(x))$, where $e:=[\varphi_{M_Q},f_Q(\varphi)]$. Observe that, by \cite[Theorem A]{Isa-Nav91} and using the fact that $S_p\leq \z(\wh{A})$, we have $\wt{f_Q(\varphi)}_{(N_P)_0}=f_{P_0}(\mu_0)$ and $\wt{f_Q(\varphi)}_{(K_P)_0}=f_{P_0}(\mu_0)^\diamond$.

Let $\wh{\mathcal{P}}'$ be an irreducible representation of $\wh{A_P}$ affording $\wt{f_Q(\varphi)}$ and consider the projective representation $\mathcal{P}'$ of $A_P$ defined by $\mathcal{P}'(x):=\wh{\mathcal{P}}'(x,1)$, for every $x\in A_P$. Notice that $\mathcal{P}'$ is associated with $f_P(\mu)^\diamond$ and that its factor set coincides with $\alpha_{A_P\times A_P}$. Furthermore, as $\c_{\wh{A}}(\wh{K})=\wh{\c_A(K)}$ and by \eqref{eq:Canonical bijection consituents}, we deduce that $\mathcal{P}_{\c_A(K)}$ and $\mathcal{P}'_{\c_A(K)}$ are associated with the same scalar function

Next, let $\vartheta=\mu^\diamond\nu\in\irr(K\mid \mu)$, with $\nu\in\irr(K/N)$, and observe that $A_\vartheta=A_\nu$. Let $\mathcal{Q}$ be a projective representation of $A_\nu$ associated with $\nu$ and notice that $\mathcal{Q}_{A_{P,\nu}}$ is a projective representation of $A_{P,\nu}$ associated with $\nu_{K_P}$. It follows that $\mathcal{S}:=\mathcal{P}_{A_\nu}\otimes \mathcal{Q}$ is a projective representation of $A_\nu$ associated with $\vartheta$, while $\mathcal{S}':=\mathcal{P}'_{A_{P,\nu}}\otimes \mathcal{Q}_{A_{P,\nu}}$ is a projective representation of $A_{P,\nu}$ associated with $\Psi_{\mu,P}(\vartheta)=f_P(\mu)^\diamond\nu_{K_P}$. We claim that $(A_\vartheta, K,\vartheta)\iso{K} (A_{P,\vartheta}, K_P, \Psi_{\mu,P}(\vartheta))$ via $(\mathcal{S},\mathcal{S}')$. By the previous paragraph, one can easily check that the group theoretical conditions hold, that the factor set of $\mathcal{S}'$ coincides with the restriction of the factor set of $\mathcal{S}$ and that $\mathcal{S}_{\c_{A_\nu}(K)}$ and $\mathcal{S}'_{\c_{A_\nu}(K)}$ are associated with the same scalar function. To conclude, it remains to check the condition on block induction. By the proof of \cite[Theorem 4.4]{Nav-Spa14I} it's enough to show that
$$\left(\frac{|K|_{p'}\tr(\mathcal{S}(x))}{p^{\h(\vartheta)}\vartheta(1)_{p'}}\right)^*=\left(\frac{|K_P|_{p'}\tr(\mathcal{S}'(x))}{p^{\h(\Psi_{\mu,P}(\vartheta))}\Psi_{\mu,P}(\vartheta)(1)_{p'}}\right)^*,$$
for every $p$-regular $x\in A_{P,\vartheta}$ such that $P\in\syl_p(\c_K(x))$. Fix a $p$-regular element $x\in A_{P,\vartheta}$ with $P\in\syl_p(\c_K(x))$. Then $Q\in\syl_p(\c_{\wh{K}}(x,1))$ and \eqref{eq:Canonical bijection values} implies
$$\tr\left(\mathcal{S}(x)\right)^*=\wt{\varphi}(x,1)^*\tr(\mathcal{Q}(x))^*=\left(e\wt{f_Q(\varphi)}(x,1)\right)^*\tr(\mathcal{Q}(x))^*=e^*\tr\left(\mathcal{S}'(x)\right)^*.$$
As $e=[\mu_{N_P},f_P(\mu)]$ and by \cite[Theorem 5.2 (b)]{Nav-Spa14I}, we obtain 
$$\vartheta(1)_{p'}=\mu(1)\equiv [\mu_{N_P},f_P(\mu)]|N:N_P|f_P(\mu)(1)\equiv e\frac{|K|_{p'}}{|K_P|_{p'}}\Psi_{\mu,P}(\vartheta)(1)_{p'}\pmod{p}$$
and therefore
$$\left(\frac{|K|_{p'}\tr(\mathcal{S}(x))}{p^{\h(\vartheta)}\vartheta(1)_{p'}}\right)^*=\left(\frac{e|K|_{p'}\tr(\mathcal{S}'(x))}{\nu(1)\vartheta(1)_{p'}}\right)^*=\left(\frac{|K_P|_{p'}\tr(\mathcal{S}'(x))}{p^{\h(\Psi_{\mu,P}(\vartheta))}\Psi_{\mu,P}(\vartheta)(1)_{p'}}\right)^*.$$
Now the proof is complete.
\end{proof}

As a consequence, applying Proposition \ref{prop:Constructing bijections over bijections} and Proposition \ref{prop:Above the Glauberman correspondence with block relation}, for $N\leq J\leq A$ we obtain a defect preserving $A_{P,J}$-equivariant bijection
$$\Phi:\irr(J\mid \mu)\to \irr(J_P\mid f_P(\mu))$$
such that
$$\left(A_{J,\chi}, J, \chi\right)\iso{J}\left(A_{J,P,\chi},J_P,\Phi(\chi)\right),$$
for every $\chi\in\irr(J\mid \mu)$. Finally, we obtain the main result of this section by considering a normal $p$-chain $\d$ with last term $P$ and $J=NG_\d$.

\begin{theo}
\label{thm:Really above the Glauberman correspondence with block relation with chains}
Let $N\leq G\unlhd A$, with $N\unlhd A$ a $p'$-subgroup, and consider a normal $p$-chain $\d$ of $G$ with final term $P$. Let $\mu\in\irr_A(N)$ and $f_P(\mu)\in\irr(N_P)$ be its $P$-Glauberman correspondent. Then there exists a defect preserving $A_\d$-equivariant bijection
$$\Phi_{\mu,\d}:\irr(NG_\d\mid \mu)\to\irr(G_\d\mid f_P(\mu))$$
such that
$$\left(NA_{\d,\chi},NG_\d,\chi\right)\iso{G}\left(A_{\d,\chi},G_\d,\Phi_{\mu,\d}(\chi)\right),$$
for every $\chi\in\irr(NG_\d\mid \mu)$.
\end{theo}

\begin{proof}
Let $K:=NP$ and observe that, without loss of generality, we may assume $K\unlhd A$. Now the result follows from Proposition \ref{prop:Constructing bijections over bijections}, Proposition \ref{prop:Above the Glauberman correspondence with block relation} and Lemma \ref{lem:Relation on triple wrt different groups}.
\end{proof}

\section{$N$-block isomorphic character triples and Fong correspondence}
\label{sec:Fong}

In this section, we show that the Fong correspondence \cite{Fon61} can be used to construct $N$-block isomorphic character triples. For completeness, we state the Fong correspondence in the form we need.

\begin{hyp}
\label{hyp:Fong correspondence}
Let $N$ be a normal $p'$-subgroup of $A$ and consider $\mu\in\irr_A(N)$. Let $\mathcal{P}\in{\proj(A\mid \alpha)}$ be a projective representation associated with $(A,N,\mu)$ such that $\alpha(x,y)^{|N|^2}=1$, for every $x,y\in A$ (see \cite[Theorem 3.5.7]{Nag-Tsu89}), and denote by $\wh{A}$ the $p'$-central extension of $A$ by $S:=\langle\alpha(x,y)\mid x,y\in A\rangle$ defined by $\mathcal{P}$ (see \cite[Section 5.3]{Nav18}). Let $\epsilon:\wh{A}\to A$ be the epimorphism given by $\epsilon(x,s):=x$, for every $x\in A$ and $s\in S$, and consider $N_0:=\{(n,1)\mid n\in N\}\unlhd \wh{A}$. For every $X\leq A$, set $\wh{X}:=\epsilon^{-1}(X)$ and $\wt{X}:=\wh{X}N_0/N_0$. Consider the irreducible representation $\wh{\mathcal{P}}$ of $\wh{A}$ defined by $\wh{\mathcal{P}}(x,s):=s\mathcal{P}(x)$, for every $x\in A$ and $s\in S$, and denote its character by $\tau$. Let $\wh{\lambda}\in\irr(\wh{N})$ be the linear character defined by $\wh{\lambda}(n,s):=s^{-1}$, for every $n\in N$ and $s\in S$, and set $\wh{\mu}:=\mu_0\times 1_S\in\irr(\wh{N})$, where $\mu_0$ correspond to $\mu$ via the isomorphism $N\simeq N_0$. Notice that $\tau$ extends $\wh{\mu}\wh{\lambda}^{-1}$. Finally, denote by $\wt{\mu}$ the character $\wh{\lambda}$ viewed as a character of $\wt{N}=\wh{N}/N_0$, that is $\wt{\mu}(N_0(n,s))=s^{-1}$, for every $n\in N$ and $s\in S$.
\end{hyp}

\begin{theo}[Fong]
\label{thm:Fong correspondence}
Assume Hypothesis \ref{hyp:Fong correspondence}. If $N\leq H\leq A$, then:
\begin{enumerate}
\item $\wt{H}$ is a $p'$-central extension of $H/N$ by the central $p'$-subgroup $\wt{N}\simeq S$;
\item There exists a bijection
\begin{align*}
\Bl(H\mid \bl(\mu))&\to \Bl(\wt{H}\mid \bl(\wt{\mu}))
\\
B &\mapsto \wt{B}
\end{align*}
\item Let $D\in\delta(B)$ and consider $Q\in\syl_p(\wh{D})$ so that $\wh{D}=Q\times S$. Then $QN_0/N_0\in\delta(\wt{B})$. In particular $B$ and $\wt{B}$ have isomorphic defect groups;
\item For every $B\in\Bl(H\mid \bl(\mu))$ corresponding to $\wt{B}\in\Bl(\wt{H}\mid \bl(\wt{\mu}))$ via the bijection in (ii), there exists a defect preserving bijection
\begin{align*}
\irr(B)&\to \irr(\wt{B})
\\
\psi &\mapsto \wt{\psi}
\end{align*}
such that, if $\wh{\psi}$ is the inflation to $\wh{H}$ of the character of $\wh{H}/S\simeq H$ corresponding to $\psi$ and $\wt{\psi}'$ is the inflation to $\wh{H}$ of $\wt{\psi}$, then $\wh{\psi}=\tau_{\wh{H}}\wt{\psi}'$;
\item For $\wh{x}\in \wh{A}$ set $x:=\epsilon(\wh{x})$ and $\wt{x}:=N_0\wh{x}$. Then $\wt{\psi^x}=\left(\wt{\psi}\right)^{\wt{x}}$ and $\wt{B^x}=\left(\wt{B}\right)^{\wt{x}}$, for every $B\in{\Bl(G\mid \bl(\mu))}$ and $\psi\in\irr(B)$.
\end{enumerate}
\end{theo}

\begin{proof}
Consider $\psi\in\irr(H\mid \mu)$ afforded by $\mathfrak{X}$. We just show how to construct $\wt{\psi}$. By \cite[Theorem 10.11]{Nav18}, there exists an irreducible projective representation $\mathcal{Q}\in\proj(H/N\mid \alpha_{H\times H}^{-1})$ such that $\mathfrak{X}=\mathcal{Q}\otimes \mathcal{P}_H$ unique up to similarity. Now, $\wh{\mathcal{Q}}(x,s):=\mathcal{Q}(x)s^{-1}$, for every $x\in H$ and $s\in S$, defines an irreducible linear representation of $\wh{H}$ with $N_0\leq \ker(\wh{\mathcal{Q}})$ and whose character lies over $\wh{\lambda}$. If we consider the inflation $\wh{\mathfrak{X}}$ to $\wh{H}$ of the representation of $\wh{H}/S\simeq H$ corresponding to $\mathfrak{X}$, that is $\wh{\mathfrak{X}}(\wh{x}):=\mathfrak{X}(\epsilon(\wh{x}))$, for every $\wh{x}\in \wh{H}$, then $\wh{\mathfrak{X}}=\wh{\mathcal{Q}}\otimes \wh{\mathcal{P}}_{\wh{H}}$. Define $\wt{\mathfrak{X}}$ to be the irreducible representation of $\wt{H}=\wh{H}/N_0$ whose inflation is $\wh{\mathcal{Q}}$, and let $\wt{\psi}$ be the character afforded by $\wt{\mathfrak{X}}$. Then $\wt{\psi}\in\irr(\wt{A}\mid \wt{\mu})$ and, if $\wh{\psi}$ is the inflation to $\wh{H}$ of the character of $\wh{H}/S\simeq H$ corresponding to $\psi$ and $\wt{\psi}'$ is the inflation of $\wt{\psi}$ to $\wh{H}$, then $\wh{\psi}=\tau_{\wh{H}}\wt{\psi}'$. The result follows from \cite{Fon61}. The description of defect groups is a consequence of the proof of \cite[2C]{Fon61}. To conclude, for $\wh{x}\in \wh{A}$, set $x:=\epsilon(\wh{x})$ and $\wt{x}:=N_0\wh{x}$. Then $\wh{\psi^x}=(\wh{\psi})^{\wh{x}}=(\wt{\psi}'\tau_{\wh{H}})^{\wh{x}}=(\wt{\psi}')^{\wh{x}}\tau_{\wh{H}}=(\wt{\psi}^{\wt{x}})'\tau_{\wh{H}}$, where $\wh{\psi^x}$ is the inflation to $\wh{H}$ of the character of $\wh{H}/S\simeq H$ corresponding to $\psi^x$ and $(\wt{\psi}^{\wt{x}})'$ is the inflation of $\wt{\psi}^{\wt{x}}$ to $\wh{H}$. Thus $(\wt{\psi})^{\wt{x}}$ coincides with $\wt{\psi^x}$ the Fong correspondent of $\psi^x$. In particular, since $\wt{\psi^x}\in \irr(\wt{B^x})$ and $(\wt{\psi})^{\wt{x}}\in\irr(\wt{B}^{\wt{x}})$, we conclude that $\wt{B^x}=\wt{B}^{\wt{x}}$.
\end{proof}

In the situation of Theorem \ref{thm:Fong correspondence}, we refer to $\wt{B}$ as the Fong correspondent of $B$ and to $\wt{\psi}$ as the Fong correspondent of $\psi$. An important feature of the Fong correspondence is that it is compatible with block induction.

\begin{prop}
\label{prop:Fong correspondence and block induction}
Assume Hypothesis \ref{hyp:Fong correspondence} and let $N\leq X\leq Y\leq A$. Let $b\in\Bl(X\mid \bl(\mu))$ with Fong correspondent $\wt{b}\in\Bl(\wt{X}\mid \bl(\wt{\mu}))$ and suppose that the induced blocks $b^Y$ and $(\wt{b})^{\wt{Y}}$ are defined. Then $\wt{b^Y}=(\wt{b})^{\wt{Y}}$.
\end{prop}

\begin{proof}
This result has been shown in \cite{Rob00}. It can also be deduced from \cite[Theorem 14.3]{Dad94}.
\end{proof}

For $x\in G$, we denote by $\cl_G(x)$ the $G$-conjugacy class of $x$. Moreover, for any subset $K$ of $G$, we denote by $K^+$ the sum of its elements in the group algebra of $G$.

\begin{theo}
\label{thm:Fong correspondence and character triple relations}
Assume Hypothesis \ref{hyp:Fong correspondence}. For $i=1,2$, consider $N\leq L_i\unlhd H_i\leq A$ and a $H_i$-invariant $\psi_i\in\irr(L_i\mid \mu)$. Notice that $\wt{L}_i\unlhd \wt{H}_i$ and that the Fong correspondent $\wt{\psi}_i\in\irr(\wt{L}_i\mid \wt{\mu})$ is $\wt{H}_i$-invariant. Let $L_i\leq G\unlhd A$ and assume
$$\left(\wt{H}_1,\wt{L}_1,\wt{\psi}_1\right)\iso{\wt{G}}\left(\wt{H}_2,\wt{L}_2,\wt{\psi}_2\right).$$
Then
$$\left(H_1,L_1,\psi_1\right)\iso{G}\left(H_2,L_2,\psi_2\right).
$$
\end{theo}

\begin{proof}
The group theoretical conditions are clearly satisfied and without loss of generality we may assume $A=GH_i$, $\wh{A}=\wh{G}\wh{H}_i$ and $\wt{A}=\wt{G}\wt{H}_i$. Consider $B_i:=\bl(\psi_i)$ and its Fong correspondent $\wt{B}_i=\bl(\wt{\psi}_i)$. By hypothesis, there exists a defect group $D_i\in\delta(\wt{B}_i)$ such that $\c_{\wt{A}}(D_i)\leq \wt{H}_i$. Furthermore, by Theorem \ref{thm:Fong correspondence} (iii) we can find a defect group $P_i\in\delta(B_i)$ such that, if $Q_i\in\syl_p(\wh{P}_i)$, then $D_i=Q_iN_0/N_0$. In particular 
\begin{equation}
\label{eq:Fong 1}
\c_{\wh{A}}(Q_i)\leq \wh{H}_i
\end{equation}
and, noticing that
$$\epsilon\left(\c_{\wh{A}}(Q_i)\right)=\c_A(P_i),$$
we obtain $\c_A(P_i)\leq H_i$. Fix projective representations $(\wt{\mathcal{R}}_1,\wt{\mathcal{R}}_2)$ associated with $\left(\wt{H}_1,\wt{L}_1,\wt{\psi}_1\right)\iso{\wt{G}}\left(\wt{H}_2,\wt{L}_2,\wt{\psi}_2\right)$ and let $\wt{\alpha}_i$ be the factor set of $\wt{\mathcal{R}}_i$. Consider a projective representation $\mathcal{R}_i\in\proj(H_i\mid \alpha_i)$ associated with $\psi_i$ and define the projective representation $\wh{\mathcal{R}}_i\in\proj(\wh{H}_i\mid \wh{\alpha}_i)$ given by 
$$\wh{\mathcal{R}}_i(h):=\mathcal{R}_i(\epsilon(h)),$$
for every $h\in \wh{H}_i$. Notice that $\wh{\alpha}_i(h,k)=\alpha_i(\epsilon(h),\epsilon(k))$, for all $h,k\in \wh{H}_i$, and that $\wh{\mathcal{R}}_i$ is associated with $\wh{\psi}_i$. Let $\wt{\mathcal{R}}_i'\in\proj(\wh{H}_i\mid \wt{\alpha}'_i)$ be the projective representation defined by 
$$\wt{\mathcal{R}}_i'(h):=\wt{\mathcal{R}}_i(N_0h),$$
for every $h\in \wh{H}_i$. Clearly $\wt{\alpha}_i'(h,k)=\wt{\alpha}_i(N_0h,N_0k)$, for all $h,k\in \wh{H}_i$, and $\wt{\mathcal{R}}_i'$ is associated with $\wt{\psi}_i'$. As $\wh{\mathcal{R}}_i$ and $\wh{\mathcal{P}}_{\wh{H}_i}\otimes \wt{\mathcal{R}}_i'$ are a projective representations of $\wh{H}_i$ associated with $\tau_{\wh{L}_i}\wt{\psi}_i'=\wh{\psi}_i$, there exists a map $\wh{\xi}_i:\wh{H}_i/\wh{L}_i\to\mathbb{C^\times}$ such that $\wh{\xi}_i\wh{\mathcal{R}}_i=\wh{\mathcal{P}}_{\wh{H}_i}\otimes \wt{\mathcal{R}}_i'$. Let $\xi_i:H_i/L_i\to \mathbb{C}^\times$ corresponds to $\wh{\xi}_i$ via the isomorphism $H_i/L_i\simeq \wh{H}_i/\wh{L}_i$. Replacing $\mathcal{R}_i$ with $\xi_i\mathcal{R}_i$, we may assume
\begin{equation}
\label{eq:Fong 2}
\wh{\mathcal{R}}_i=\wh{\mathcal{P}}_{\wh{H}_i}\otimes \wt{\mathcal{R}}_i'.
\end{equation}
Now, as the factor sets $\wt{\alpha}_1$ and $\wt{\alpha}_2$ coincide under the isomorphism $\wt{H}_1/\wt{L}_1\simeq \wt{H}_2/\wt{L}_2$, we deduce that $\alpha_1$ and $\alpha_2$ coincide under the isomorphism $H_1/L_1\simeq H_2/L_2$. By hypothesis $\wt{\mathcal{R}}_1$ and $\wt{\mathcal{R}}_2$ define the same scalar function on $\c_{\wt{A}}(\wt{G})$. As $\c_{\wh{A}}(\wh{G})N_0/N_0\leq \c_{\wt{A}}(\wt{G})$ and $\c_{\wh{A}}(\wh{G})\leq \wh{H}_1\cap \wh{H}_2$ by \eqref{eq:Fong 1}, the scalar functions defined by $\wt{\mathcal{R}}'_1$ and $\wt{\mathcal{R}}'_2$ on $\c_{\wh{A}}(\wh{G})$ coincide. Now $\wh{\mathcal{R}}_{1,\c_{\wh{A}}(\wh{G})}$ and $\wh{\mathcal{R}}_{2,\c_{\wh{A}}(\wh{G})}$ are associated with the same scalar function and, since $\epsilon(\c_{\wh{A}}(\wh{G}))=\c_A(G)$ (see \cite[Theorem 4.1 (d)]{Nav-Spa14I}), the same is true for $\mathcal{R}_{1,\c_A(G)}$ and $\mathcal{R}_{2,\c_A(G)}$.

Next, consider $G\leq J\leq A$ and set $J_i:=J\cap H_i$. Notice that, if $\chi\in\irr(J_1\mid \psi_1)$, then Theorem \ref{thm:Fong correspondence} (iv) implies that $\wt{\chi}\in\irr(\wt{J}_1\mid \wt{\psi}_1)$. Write $\chi=\tr(\mathcal{Q}_{J_1}\otimes \mathcal{R}_{1,J_1})$, for some $\mathcal{Q}\in\proj(J/G)$. If we set $\wh{Q}(x):=\mathcal{Q}(\epsilon(x))$ for every $x\in \wh{J}$, then \eqref{eq:Fong 2} implies
\begin{align*}
\wh{\chi}_1&=\tr\left(\wh{\mathcal{Q}}_{\wh{J}_1}\otimes \wh{\mathcal{R}}_{1,\wh{J}_1}\right)
\\
&=\tr\left(\wh{\mathcal{Q}}_{\wh{J}_1}\otimes \wt{\mathcal{R}}'_{1,\wh{J}_1} \otimes \wh{\mathcal{P}}_{\wh{J}_1}\right)
\end{align*}
and therefore $\wt{\chi}'=\tr(\wh{\mathcal{Q}}_{\wh{J}_1}\otimes \wt{\mathcal{R}}'_{1,\wh{J}_1})$. Now, let $\wt{\mathcal{Q}}\in\proj(\wt{J}/\wt{G})$ correspond to $\wh{\mathcal{Q}}$ via the isomorphism $\wt{J}/\wt{G}\simeq \wh{J}/\wh{G}$ and observe that the Fong correspondent of $\chi$ can be written as $\wt{\chi}=\tr(\wt{\mathcal{Q}}_{\wt{J}_1}\otimes \wt{\mathcal{R}}_{1,\wt{J}_1})$. By definition $\wt{\sigma}_{\wt{J}_1}(\wt{\chi})=\tr(\wt{\mathcal{Q}}_{\wt{J}_2}\otimes \wt{\mathcal{R}}_{2,\wt{J}_2})$ so that its inflation $\wt{\sigma}_{\wt{J}_1}(\wt{\chi})'=\tr(\wh{\mathcal{Q}}_{\wh{J}_2}\otimes \wt{\mathcal{R}}'_{2,\wh{J}_2})$. By Theorem \ref{thm:Fong correspondence} (iv) and \eqref{eq:Fong 2} we obtain
\begin{align*}
\tau_{\wh{J}_2}\wt{\sigma}_{\wt{J}_1}(\wt{\chi})'&=\tr\left(\wh{\mathcal{P}}_{\wh{J}_2}\otimes\wh{\mathcal{Q}}_{\wh{J}_2}\otimes \wt{\mathcal{R}}'_{2,\wh{J}_2}\right)
\\
&=\tr\left(\wh{\mathcal{Q}}_{\wh{J}_2}\otimes \wh{\mathcal{R}}_{2,\wh{J}_2}\right)
\\
&=\wh{\sigma_{J_1}(\chi)}
\\
&=\tau_{\wh{J}_2}\wt{\sigma_{J_1}(\chi)}'
\end{align*}
and therefore
$$\wt{\sigma_{J_1}(\chi)}=\wt{\sigma}_{\wt{J}_1}\left(\wt{\chi}\right).$$
Since by hypothesis $\bl(\wt{\sigma_{J_1}(\chi)})^{\wt{J}}=\bl(\wt{\chi})^{\wt{J}}$, we conclude from Proposition \ref{prop:Fong correspondence and block induction} that $\bl(\sigma_{J_1}(\chi))^J=\bl(\chi)^J$. This completes the proof.
\end{proof}

From now on we consider $N\leq G\unlhd A$. Since $\wt{N}$ is a central $p'$-subgroup of $\wt{G}$, for every $p$-subgroup $P$ of $G$ we have a decomposition $\wt{P}=\wt{N}\times \O_p(\wt{P})$. We write $\wt{P}_p:=\O_p(\wt{P})$. Mapping $P$ to $\wt{P}_p$ induces a length preserving bijection
\begin{align}
\mathfrak{N}(G,Z)/G &\to\mathfrak{N}(\wt{G},\wt{Z}_p)/\wt{G} \label{eq:Bijection on chains}
\\
\d &\mapsto \wt{\d} \nonumber
\end{align}
which commutes with the action of $A$ and $\wt{A}$. In particular, observe that $\wt{NG_\d}=\wt{G}_{\wt{\d}}$. Using Theorem \ref{thm:Really above the Glauberman correspondence with block relation with chains}, Theorem \ref{thm:Fong correspondence} and Theorem \ref{thm:Fong correspondence and character triple relations} we obtain the following corollaries.

\begin{cor}
\label{cor:Fong correspondence and Glauberman bijection}
Assume Hypothesis \ref{hyp:Fong correspondence} and let $N\leq G\unlhd A$. Consider a normal $p$-chain $\d$ of $G$ with final term $P$ and let $f_P(\mu)\in\irr(N_P)$ be the $P$-Glauberman correspondent of $\mu$. Then there exists a defect preserving bijection
$$\Gamma_{\mu,\d}:\irr\left(G_\d\enspace\middle|\enspace f_P(\mu)\right)\to\irr\left(\wt{G}_{\wt{\d}}\enspace\middle|\enspace \wt{\mu}\right)$$
commuting with the action of $A$ and $\wt{A}$.
\end{cor}

\begin{proof}
This follows immediately by Theorem \ref{thm:Really above the Glauberman correspondence with block relation with chains} and Theorem \ref{thm:Fong correspondence}.
\end{proof}

The bijections described in the previous corollary are compatible with the relation $\iso{N}$ from Definition \ref{def:N-block isomorphic character triples}.

\begin{cor}
\label{cor:Fong correspondence and Glauberman bijection with character triple relations}
Assume Hypothesis \ref{hyp:Fong correspondence} and let $N\leq G\unlhd A$. Consider normal $p$-chains $\d$ and $\e$ of $G$ with final term respectively $P$ and $Q$ and let $\Gamma_{\mu,\d}$ and $\Gamma_{\mu,\e}$ be the corresponding bijections given by Corollary \ref{cor:Fong correspondence and Glauberman bijection}. Let $\vartheta\in\irr(G_\d\mid f_P(\mu))$ and $\chi\in\irr(G_\e\mid f_Q(\mu))$ and suppose that
$$\left(\wt{A}_{\wt{\d},\Gamma_{\mu,\d}(\vartheta)},\wt{G}_{\wt{\d}},\Gamma_{\mu,\d}(\vartheta)\right)\iso{\wt{G}}\left(\wt{A}_{\wt{\e},\Gamma_{\mu,\e}(\chi)},\wt{G}_{\wt{\e}},\Gamma_{\mu,\e}(\chi)\right).$$
Then
$$\left(A_{\d,\vartheta},G_\d,\vartheta\right)\iso{G}\left(A_{\e,\chi},G_\e,\chi\right).$$
\end{cor}

\begin{proof}
This is a consequence of Theorem \ref{thm:Really above the Glauberman correspondence with block relation with chains}, Theorem \ref{thm:Fong correspondence and character triple relations} and Corollary \ref{cor:Fong correspondence and Glauberman bijection}.
\end{proof}

The result that we are actually going to need in the final proof is the following. This is obtained by putting together all the results obtained so far.

\begin{cor}
\label{cor:CTC for p-solvable in the stable case}
Assume Hypothesis \ref{hyp:Fong correspondence} and let $N\leq G\unlhd A$. Let $Z$ be a central $p$-subgroup of $G$ and consider a block $B\in\Bl(G\mid \bl(\mu))$ whose defect groups are larger than $Z$. Then the Fong correspondent $\wt{B}\in\Bl(\wt{G})$ has defect groups larger than $\wt{Z}_p$ and there exists a bijection
$$\Delta:\C^d(B,Z)/G\to\C^d(\wt{B},\wt{Z}_p)/\wt{G}$$
that preserves the length of the $p$-chains, commutes with the ation of $A$ and $\wt{A}$ and such that, if
$$\left(\wt{A}_{(\wt{\d},\wt{\vartheta})}, \wt{G}_{\wt{\d}}, \wt{\vartheta}\right)\iso{\wt{G}}\left(\wt{A}_{(\wt{\e},\wt{\chi})}, \wt{G}_{\wt{\d}}, \wt{\chi}\right),$$
then
$$\left(A_{(\d,\vartheta)},G_{\d},\vartheta\right)\iso{G}\left(A_{(\e,\chi)},G_{\e},\chi\right),$$
for every $(\d,\vartheta), (\e,\chi)\in\C^d(B,Z)$, $(\wt{\d},\wt{\vartheta})\in\Delta(\overline{(\d,\vartheta)})$ and $(\wt{\e},\wt{\chi})\in\Delta(\overline{(\e,\chi)})$.
\end{cor}

\begin{proof}
Let $\d\in\mathfrak{N}(G,Z)$ with last term $P$ and consider $\wt{\d}\in\mathfrak{N}(\wt{G},\wt{Z}_p)$. If $\vartheta\in\irr(G_\d)$ and $\bl(\vartheta)^G=B$, then $\vartheta$ lies over $f_P(\mu)$ by Corollary \ref{cor:Harris-Knorr and p-chains}. Now, there exists a unique $\psi\in\irr(NG_\d\mid \mu)$ such that $\vartheta=\Phi_{\mu,\d}(\psi)$ and $\Gamma_{\mu,\d}(\vartheta)=\wt{\psi}$ is the Fong correspondent of $\psi$. By Theorem \ref{thm:Really above the Glauberman correspondence with block relation with chains}, we know that $\bl(\vartheta)^{NG_\d}=\bl(\psi)$, hence $\bl(\vartheta)^G=B$ if and only if $\bl(\psi)^G=B$. Furthermore, by Proposition \ref{prop:Fong correspondence and block induction} it follows that $\bl(\psi)^G=B$ if and only if $\bl(\wt{\psi})^{\wt{G}}=\wt{B}$. This shows that the set of characters of $G_\d$ whose block induces to $B$ is mapped via $\Gamma_{\mu,\d}$ to the set of characters of $\wt{G}_{\wt{\d}}$ whose block induces to $\wt{B}$. We define
$$\Delta\left(\overline{(\d,\vartheta)}\right):=\overline{\left(\wt{\d},\Gamma_{\mu,\d}(\vartheta)\right)},$$ 
for every $(\d,\vartheta)\in\C^d(B,Z)$. By \eqref{eq:Bijection on chains}, Corollary \ref{cor:Fong correspondence and Glauberman bijection} and Corollary \ref{cor:Fong correspondence and Glauberman bijection with character triple relations} we conclude that $\Delta$ is a bijection with the required properties.
\end{proof}

\section{Structure of a minimal counterexample}

In this section, we finally prove Sp\"ath's Character Triple Conjecture for $p$-solvable groups. Our proof is inspired by the argument developed in \cite{Rob00}. As in Robinson's work, what we are actually going to show is that a minimal counterexample $G$ to Conjecture \ref{conj:Character Triple Conjecture} satisfies $\O_p(G)\O_{p'}(G)\leq \z(G)$. Since the conjecture trivially holds for abelian groups, Theorem \ref{thm:Main theorem} will then follow as a corollary of (the proof of) Theorem \ref{thm:Minimal counterexample structure}. In this paper we consider subpairs in the sense of \cite{Ols82}, i.e. pairs $(P,b_P)$, where $P$ is a $p$-subgroup of $G$ and $b_P\in\Bl(P\c_G(P))$.

\begin{prop}
\label{prop:Minimal counterexample structure, unstable}
Assume that $G\unlhd A$ is a minimal counterexample to Conjecture \ref{conj:Character Triple Conjecture} with respect to $|G:\z(G)|$ first and then to $|A|$ and consider $Z\leq \z(G)$, $B\in\Bl(G)$ and $d\geq 0$ for which the conjecture fails to hold. Then every block $b\in\Bl(\O_{p'}(G))$ covered by $B$ is $A$-invariant.
\end{prop}

\begin{proof}
Set $N:=\O_{p'}(G)$ and fix a block $\bl(\mu)\in\Bl(N)$ covered by $B$. For every subgroup $H\leq A$, set $H^\vee:=H_\mu$. Let $B^\vee\in\Bl(G^\vee\mid \bl(\mu))$ be the Fong--Reynolds correspondent of $B$ over $\bl(\mu)$ \cite[Theorem 9.14]{Nav98}. Since $B$ and $B^\vee$ have a common defect group $D\leq G^\vee$, by \cite[Theorem 3.10]{Alp-Bro79} or \cite[Theorem 2.1]{Ols82} we can find a $B^\vee$-Sylow subpair $(D,b_D^\vee)$ such that $(D,b_D)$ is a $B$-Sylow subpair, where $b_D:=(b_D^\vee)^{D\c_G(D)}$. Notice that, by using Corollary \ref{cor:Harris-Knorr} together with the theory of subpairs, the block $b_D^\vee$ covers $\bl(f_D(\mu))$. Hence $b_D$ covers $\bl(f_D(\mu))$. More generally, the block $b_Q$ covers $\bl(f_Q(\mu))$, for every $B$-subpair $(Q,b_Q)\leq (D,b_D)$. Using this observation, we can construct an $A$-transversal $\mathbb{T}$ in $\C^d(B,Z)$ such that $P\leq G^\vee$ and $\vartheta\in\irr(G_\d\mid f_P(\mu))$, for every $(\d,\vartheta)\in\mathbb{T}$ with $P$ the last term of $\d$.

Consider $(\d,\vartheta)\in\mathbb{T}$ and let $P$ be the last term of $\d$. Notice that $G_\d^\vee=G_{\d,f_P(\mu)}$ and let $\vartheta^\vee\in\irr(G^\vee_\d\mid f_P(\mu))$ be the Clifford correspondent of $\vartheta$ over $f_P(\mu)$. As $A=GA^\vee$, we obtain an $A^\vee$-equivariant bijection
$$\Upsilon:\C^d(B,Z)/G\to \C^d(B^\vee,Z)/G^\vee$$ 
by defining $\Upsilon(\overline{(\d,\vartheta)}^y):=\overline{(\d,\vartheta^\vee)}^y$, for every $(\d,\vartheta)\in\mathbb{T}$ and $y\in A^\vee$. Since ${|G^\vee:\z(G^\vee)|}\leq {|G:\z(G)|}$, if $\mu$ is not $A$-invariant, then there exists an $A^\vee$-equivariant bijection 
$$\Omega^\vee:\mathcal{C}^d(B^\vee,Z)_+/G^\vee\to \mathcal{C}^d(B^\vee,Z)_-/G^\vee$$
such that 
$$\left(A^\vee_{(\d,\vartheta^\vee)},G^\vee_{\d},\vartheta^\vee\right)\iso{G^\vee}\left(A^\vee_{(\e,\chi^\vee)}, G^\vee_{\e},\chi^\vee\right),$$
for every $(\d,\vartheta^\vee)\in\C^d(B^\vee,Z)_+$ and $(\e,\chi^\vee)\in\Omega^\vee(\overline{(\d,\vartheta^\vee)})$. Combining $\Omega^\vee$ with $\Upsilon$ and applying Proposition \ref{prop:Irreducible induction and double relation block}, we obtain an $A$-equivariant bijection
$$\Omega:\C^d(B,Z)_+/G \to \C^d(B,Z)_-/G$$
such that
$$\left(A_{(\d,\vartheta)},G_\d,\vartheta\right)\iso{G}\left(A_{(\e,\chi)},G_{\e},\chi\right),$$
for every $(\d,\vartheta)\in\C^d(B,Z)_+$ and $(\e,\chi)\in\Omega(\overline{(\d,\vartheta)})$. This is a contradiction and therefore $\mu$ must be $A$-invariant.
\end{proof}

\begin{theo}
\label{thm:Minimal counterexample structure}
Assume that $G\unlhd A$ is a minimal counterexample to Conjecture \ref{conj:Character Triple Conjecture} with respect to $|G:\z(G)|$ first and then to $|A|$ and consider $Z\leq \z(G)$, $B\in\Bl(G)$ and $d\geq 0$ for which the conjecture fails to hold. Then $\O_p(G)\O_{p'}(G)\leq \z(G)$.
\end{theo}

\begin{proof}
Set $N:=\O_{p'}(G)$ and fix a block $\bl(\mu)\in\Bl(N)$ covered by $B$. Notice that by Lemma \ref{lem:Knorr-Robinson for Character Triple Conjecture} we must have $Z=\O_p(G)$. Thus it's enough to show that $N$ is contained in the center. By Proposition \ref{prop:Minimal counterexample structure, unstable}, we know that $\mu$ is $A$-invariant and therefore we can apply the results obtained in Section \ref{sec:Fong}. Let $\wt{B}\in\Bl(\wt{G})$ be the Fong correspondent of $B$. Since $\wt{N}\leq \z(\wt{G})$, if $N\nleq \z(G)$, then ${|\wt{G}:\z(\wt{G})|}\leq {|\wt{G}:\wt{N\z(G)}|}={|G:N\z(G)|}<{|G:\z(G)|}$ and we obtain an $\wt{A}$-equivariant bijection
$$\wt{\Omega}:\C^d(\wt{B},\wt{Z}_p)_+/\wt{G}\to \C^d(\wt{B},\wt{Z}_p)_-/\wt{G}$$
such that 
$$\left(\wt{A}_{(\wt{\d},\wt{\vartheta})}, \wt{G}_{\wt{\d}}, \wt{\vartheta}\right)\iso{\wt{G}}\left(\wt{A}_{(\wt{\e},\wt{\chi})}, \wt{G}_{\wt{\d}}, \wt{\chi}\right),$$
for every $(\wt{\d},\wt{\vartheta})\in\C^d(\wt{B},\wt{Z}_p)_+$ and $(\wt{\e},\wt{\chi})\in\wt{\Omega}(\overline{(\wt{\d},\wt{\vartheta})})$. Combining $\wt{\Omega}$ with the bijection $\Delta$ given by Corollary \ref{cor:CTC for p-solvable in the stable case}, we obtain an $A$-equivariant bijection
$$\Omega:\mathcal{C}^d(B,Z)_+/G\to\mathcal{C}^d(B,Z)_-/G$$
such that
$$\left(A_{(\d,\vartheta)},G_\d,\vartheta\right)\iso{G}\left(A_{(\e,\chi)},G_{\e},\chi\right),$$
for every $(\d,\vartheta)\in\mathcal{C}^d(B,Z)_+$ and $(\e,\chi)\in\Omega(\overline{(\d,\vartheta)})$. This contradiction shows that $N$ must be contained in the center of $G$.
\end{proof}

Next, we consider the residue of characters. We are going to obtain Theorem \ref{thm:Main theorem residues} as a consequence of an analogous study of a minimal counterexample. We need the following result whose proof can be deduced by Glesser's paper \cite{Gle07}.

\begin{lem}
\label{lem:Residues}
Let $\Gamma_{\mu,\d}$ be the bijection of Corollary \ref{cor:Fong correspondence and Glauberman bijection}. Then
$$r\left(\Gamma_{\mu,\d}(\vartheta)\right)|N|\equiv \pm\mu(1)r(\vartheta)|\wt{N}|\pmod{p},$$
for every $\vartheta\in\irr(G_\d\mid f_P(\mu))$.
\end{lem}

\begin{proof}
This follows from similar computations as the ones in the proofs of \cite[Corollary 3.4 and Theorem 3.8]{Gle07}.
\end{proof}

For completeness, we state the Isaacs-Navarro refinement of the Character Triple Conjecture.

\begin{conj}[Isaacs-Navarro refinement of the Character Triple Conjecture]
\label{conj:Character Triple Conjecture with residues}
There exists a bijection $\Omega$ as in Conjecture \ref{conj:Character Triple Conjecture} such that
$$r(\vartheta)\equiv\pm r(\chi)\pmod{p},$$
for every $(\d,\vartheta)\in\C^d(B,Z)_+$ and $(\e,\chi)\in\Omega(\overline{(\d,\vartheta)})$.
\end{conj}

Finally, using the proof of Theorem \ref{thm:Minimal counterexample structure} and Lemma \ref{lem:Residues} we obtain a similar structure theorem for a minimal counterexample of Conjecture \ref{conj:Character Triple Conjecture with residues}.

\begin{theo}
\label{thm:Minimal counterexample structure with residues}
Assume that $G\unlhd A$ is a minimal counterexample to Conjecture \ref{conj:Character Triple Conjecture with residues} with respect to $|G:\z(G)|$ first and then to $|A|$ and consider $Z\leq \z(G)$, $B\in\Bl(G)$ and $d\geq 0$ for which the conjecture fails to hold. Then $\O_p(G)\O_{p'}(G)\leq \z(G)$.
\end{theo}

\begin{proof}
Set $N:=\O_{p'}(G)$ and fix a block $\bl(\mu)\in\Bl(N)$ covered by $B$. By the proof of Lemma \ref{lem:Knorr-Robinson for Character Triple Conjecture} we know that $Z=\O_p(G)$ and it's enough to show that $N\leq \z(G)$. Proceeding as in the proof of Proposition \ref{prop:Minimal counterexample structure, unstable} and noticing that induction of characters preserves the residue of characters, we deduce that $\mu$ must be $A$-invariant. Then, using Lemma \ref{lem:Residues} and adapting the the proof of Theorem \ref{thm:Minimal counterexample structure}, we obtain $N\leq \z(G)$.
\end{proof}

\bibliographystyle{alpha}
\bibliography{References}

BU WUPPERTAL, GAUSSSTR. $20$, $42119$ WUPPERTAL, GERMANY

\textit{Email address:} \href{mailto:rossi@uni.wuppertal.de}{rossi@uni.wuppertal.de}

\end{document}